\newcommand{\SQF} {David-Fuchsian b-group}

\newcommand{\Bel}{\mathrm{Bel}}
\newcommand{\Bers}{\mathrm{Bers}}

\newcommand{\bel}{\mathrm{bel}}

\newcommand{\Cl}{\mathrm{Cl}}

\newcommand{\Cyl}{\operatorname{Cyl}}

\newcommand{\Hol}{\mathrm{Hol}}
\newcommand{\id} {\operatorname{id}}

\newcommand{\IM}{\operatorname{Im}}

\newcommand{\short}{\mathrm{short}}
\newcommand{\Leb}{\mathrm{Leb}}
\newcommand{\length}{\mathrm{length}}
\newcommand{\loc}{\mathrm{loc}}

\newcommand{\supp}{\mathrm{supp} }

\newcommand{\St}{\operatorname{St}}

\newcommand{\Teich}{\operatorname{Teich}}

\newcommand{\dt}{\ dt}

\newcommand{\dzeta}{\ d \zeta}

\newcommand{\dr}{dr}

\documentclass{amsart}

\usepackage[all]{xy}
\usepackage{enumerate}
\usepackage{bm}
\usepackage{amssymb}
\usepackage[pdftex]{graphicx}

\theoremstyle{theorem}
\newtheorem{theo}{Theorem}[section]
\newtheorem*{copytheo}{Theorem}
\newtheorem{prop}[theo]{Proposition}
\newtheorem{lem}[theo]{Lemma}

\newtheorem{claim}{Claim}

\theoremstyle{definition}
\newtheorem{de}[theo]{Definition}

\theoremstyle{remark}
\newtheorem*{Rem}{Remark}
\newtheorem*{exa}{Eg.}

\numberwithin{equation}{section}

\begin{document}

\title{New degeneration phenomenon for infinite-type Riemann surfaces}

\author{Ryo Matsuda}
\address{Department of Mathematics, Faculty of Science, Kyoto University, Kyoto 606-8502, Japan}
\email{matsuda.ryou.82c@st.kyoto-u.ac.jp}
\thanks{The author was supported in part by JST, the establishment of the University Fellowship Program for the Creation of Innovation in Science and Technology, Grant Number JPMJFS2123 }

%    General info

\date{September 15, 2024}


\newcommand{\Addresses}{{% additional braces for segxxxregating \footnotesize
  \bigskip
  \footnotesize

  Ryo Matsuda, \par\nopagebreak
  \textsc{Department of Mathematics, Faculty of Science, Kyoto University, Kyoto 606-8502, Japan}\par\nopagebreak
  \textit{E-mail} : \texttt{matsuda.ryou.82c@st.kyoto-u.ac.jp}
}}

%\setlength{\baselineskip}{9mm}
%\setlength{\abovedisplayskip}{0mm}
%\setlength{\belowdisplayskip}{0mm}

%=====================================%
%                 contents                         
%=====================================%
\maketitle
%\setcounter{tocdepth}{2}
%\tableofcontents
%\newpage

%=====================================%
%                         body                    
%=====================================%t

%\input{contents/}
\begin{abstract}
Since the Teichm\"uller space of a surface $R$ is a deformation space of complex structures defined on $R$, its Bers boundary describes the degeneration of complex structures in a certain sense. 
In this paper, constructing a concrete example, we prove that if S is a Riemann surface of infinite type, a Riemann surface with the marking exists, which is homeomorphic to the surface $R$ in the Bers boundary. We also show that many such degenerations exist in the Bers boundary. 
\end{abstract}

\section{Introduction}

The Teichm\"uller space $\Teich(\Gamma)$ of a Fuchsian group $\Gamma$ is the space of quasiconformal deformations of $\Gamma$. 
In terms of the corresponding Riemann surface $R: = \mathbb H / \Gamma$, the space $\Teich(\Gamma)$ can be regarded as the deformation space of the Riemann surface $R$; it is sometimes denoted by $\Teich(R)$. 
Furthermore, $\Teich(R)$ admits a structure as a complex manifold, which reflects a lot of the properties of the Riemann surface $R$. 

As a common technique in complex analysis, one may examine the complex manifold structure of $\Teich(R)$ by understanding properties of proper holomorphic maps from the complex disk $\mathbb{D}:= \{ z \in \mathbb{C} \mid |z| < 1 \}$ to $\Teich(R)$. On the other hand, a holomorphic function from $\mathbb{D}$ to a complex manifold is characterized by the mapping induced by non-tangential limits at almost all points of the boundary $\partial \mathbb D$ according to the Fatou--Riesz theorem. 
Because the non-tangential limits are contained in the boundary of $\Teich(\Gamma)$, understanding the boundary of $\Teich(R)$ is for understanding its complex manifold structure. 
The boundary of $\Teich(R)$ describes the degeneration of the complex structures of Riemann surfaces, which, in turn, contributes to the understanding of Riemann surfaces.

Various boundaries for $\Teich(R)$ have been considered. 
We focus on the Bers boundary $\partial_{\mathrm{Bers}} \Teich ( \Gamma )$, which is obtained by embedding $\Teich(R)$ into a
bounded region of the complex Banach space of all holomorphic quadratic forms. We will denote it $B(\Gamma)$(cf. Subsection\ref{pre: TeichKlein}). 
The embedding is called the Bers embedding $\mathcal B: \Teich(\Gamma) \to B(\Gamma)$. 
The Teichm\"uller space so embedded can be seen as the space of deformations of the Fuchsian group $\Gamma$ within the category of Kleinian groups. 
When $R$ is of finite type, it has been extensively studied in relation to three-dimensional hyperbolic geometry in great detail in Bers \cite{B}, Maskit \cite{Mas}, Abikoff \cite{Abi}, and others. 
When $R$ is of finite type, the structure of the Bers boundary has been well investigated; the Bers boundary consists of Kleinian groups corresponding to Riemann surfaces that degenerate topologically\cite{B, Mas} if $R$ is compact.
When $R$ has finitely many punctures, the same can be established by modifying the claim appropriately. This implies that the degeneration of complex structures leads to the degeneration of the topological structure. 
The theorem of Bers--Maskit tells us that the degeneration of complex structures induces topological degeneration. Although topological structure is more fundamental than complex structure, breaking the complex structure breaks the base topological structure.
It is known that almost the same holds for the case where the Riemann surfaces have finitely many punctures (more details will be given in Section\ref{pre: TeichKlein}). 
From these structures of the Bers boundary, the study by Tanigawa and others has elucidated the structure of $\Hol(\mathbb{D}, \Teich(R))$, which helps in understanding the complex manifold structure of $\Teich(R)$(Tanigawa \cite{Th, Th2}, Shiga \cite{Shi1, Shi2}). 

This paper considers the case when $R$ is of infinite type. In this case, $\Teich(R)$ becomes an infinite-dimensional complex manifold. We describe a specific construction method to illustrate degeneration phenomena where the complex structure degenerates but topologically remains homeomorphic to $R$.
In this paper, we shall call the Kleinian group corresponding to such a Riemann surface a \SQF.  
A detailed definition is described in Definition \ref{cal.kleiniangroup}. 
In addition, we show that such degeneration phenomena again form infinite-dimensional complex manifolds within the Bers boundary. The quasiconformal maps are characterized as the solution of the partial differential equation called the Beltrami equation $f_{\bar z} = \mu f_z$ considered for a function $\mu \in L^\infty(\mathbb C)$ with $\| \mu \|_{L\infty} < 1$. 
Since the Teichm\"uller space is the set of all quasiconformal deformations of the Riemann surface, a deformation by the solution of the degenerate Beltrami equation is expected to exist on the Bers boundary.
``Degenerate'' means that we consider the case $\| \mu \|_{L\infty} = 1$. David \cite{D} has studied sufficient conditions on $\mu \in L^\infty(\mathbb C)$ for the degenerate Beltrami equation to have a homeomorphic solution; its condition is called David's condition. We construct the \SQF\ by constructing a deformation of the Fuchsian group $\Gamma$ to satisfy David's condition. 
To prove that the constructed David deformation of the Fuchsian group $\Gamma$ exists on the Bers boundary, it is necessary to construct a sequence in $\Teich(\Gamma)$ that converges to the point in $B(\Gamma)$ corresponding the constructed David deformation. To show convergence, we apply the estimate given by McMullen \cite[Theorem 1.2]{McM}; we impose as a condition that  $\short(R)$ is positive, where  $\short(R)$ is the infimum of lengths of essential simple closed geodesics of $R$. 
The main result is as follows:  

\begin{copytheo} [{Theorem \ref{resultball}}] 
Let $R$ be a Riemann Surface of analytically infinite type with $\short (R) > 0$ and infinitely many homotopically independent 
essential closed geodesics $\{\gamma_n^\ast\}_{n \geq 0}$ whose lengths are bounded. Then, there exists an infinite-dimensional complex manifold in $\partial_{\Bers} \Teich (R)$ which consists of \SQF.  That is, there exists a unit ball $D$ of a Banach space and a holomorphic map $F: D \to B(\Gamma)$, where $F(D)$ is contained in the Bers boundary $\partial_{\Bers} \Teich (R)$, contains only \SQF, and the dimension of the tangent space of the image is infinite. 
\end{copytheo}

Section \ref{Preliminaries} discusses some necessary results on quasiconformal maps, Kleinian groups, and David maps. In Section \ref {sectiondefcylinderfam}, we compute the variants of the injective radius when the cylinders are stretched to apply them to McMullen's inequality. In Section 4, we construct a Kleinian group using a David map, and we prove that the Kleinian group exists on $\partial_{\mathrm{Bers}} \Teich ( \Gamma )$. That is, we prove that \SQF\ exists. In Section \ref{sectionconstruction}, we consider how many \SQF s exist on $\partial_{\mathrm{Bers}} \Teich ( \Gamma )$. In conclusion, it is shown that there exist open sets $D$ of a certain infinite-dimensional Banach space and an injective holomorphic map $F: D \to B ( \Gamma )$ such that $F(D)$ is contained in the Bers boundary and consists of \SQF.

\subsection*{Acknowlagement.}
	The author would like to thank Professor Mitsuhiro Shishikura for his helpful suggestions, which greatly simplified existing arguments on David's maps. 
He also thanks Professor Hiromi Ohtake and Professor Katsuhiko Matsuzaki for many discussions and helpful advice.
He also appreciates Takuya Murayama and Yota Maeda for providing valuable comments on grammar and style. 
JST, the establishment of the University Fellowship Program for the Creation of Innovation in Science and Technology, Grant Number JPMJFS2123, supported this work.

\section{Preliminaries}\label{Preliminaries}
\subsection{Teichm\"uller theory and Kleinian groups}\label{pre: TeichKlein}

Let $R$ be a Riemann surface whose universal covering surface is the upper half plane $\mathbb H$, 
and let it be represented by a Fuchsian group $\Gamma$ 
acting on $\mathbb H$ as $R = \mathbb H / \Gamma$. 
A quasiconformal mapping $f: \mathbb H \to \mathbb H$ is called compatible with $\Gamma$
if $\Gamma^f: = f \Gamma f^{-1}$ is again a Fuchsian group, 
and the isomorphism $\Gamma \to \Gamma^f$ given by conjugation with 
$f$ is called a quasiconformal deformation of $\Gamma$. 
Note that the deformation of $\Gamma$ induces a quasiconformal deformation of the Riemann surface $R$. 
Indeed, the map $f$ induces a quasiconformal map $\tilde f: \mathbb H/ \Gamma \to \mathbb H / \Gamma'$. 
Denote by $L^\infty ( \Gamma )$ 
the complex Banach space of all bounded measurable \textit{Beltrami differentials} for $\Gamma$ 
supported on $\mathbb H$. Each element $\mu \in L^\infty ( \Gamma )$ satisfies that 
$\mu \circ \gamma \cdot \overline{ \gamma' } = \mu \cdot \gamma'$ for every $\gamma \in \Gamma$. 
Its open unit ball $\Bel ( \Gamma )$ is the space of Beltrami coefficients for $\Gamma$. 
For each element $\mu \in \Bel ( \Gamma )$, the unique quasiconformal homeomorphism 
$w^\mu : \mathbb H \to \mathbb H$ with complex dilatation $\mu$
that fixes $0 , 1$ and $\infty$, is compatible for $\Gamma$. 
Two elements $\mu$ and $\nu$ in $\Bel ( \Gamma )$ are said to be \textit{Teichm\"uller equivalent}
if $w^\mu | _{\mathbb R} = w^\nu| _{\mathbb R}$ holds. 
The \textit{Teichm\"uller space} $\Teich ( \Gamma )$ of $\Gamma$ 
is the quotient space of $\Bel ( \Gamma )$ by the Teichm\"uller equivalence relation. 

In the following, 
we review how to embed $\Teich(\Gamma)$ into a bounded domain of the Banach space $B(\Gamma)$ by a well-known method. Let $\mathbb H^\ast$ be a lower half plane. 
$B(\Gamma)$ is the complex Banach space of all $L^\infty$ hyperbolic holomorphic quadratic forms $\varphi: \mathbb H^\ast \to \hat{\mathbb C }$ 
for $\Gamma$, that is, the point $\varphi$ of $B(\Gamma)$ satisfies $\varphi \circ \gamma \cdot (\gamma)'^2 = \varphi$
for every $\gamma \in \Gamma$, and the $L^\infty$ hyperbolic norm 
$\| \varphi \|_{B(\Gamma)}: = \sup_{z \in \mathbb H^\ast} | \varphi(z) | | \IM z |^2$ is finite.
Like $\Bel(\Gamma)$, $B(\Gamma)$ is sometimes considered a 2-form on the Riemann surface $R$ through the covering map $\pi: \mathbb H \to R$. It is denoted as $B(R)$ in such a sense.
For each $\mu \in \Bel(\Gamma)$, set 
	\[
		\tilde { \mu } =  
		\begin{cases}
		\mu & \text {on} \ \mathbb H, \\
		0 & \text {on} \ \mathbb H^\ast.
		\end{cases}
	\]
Then, there exists an unique quasiconformal mapping $f_\mu: \hat{\mathbb C } \to \hat{\mathbb C }$ such that
the beltrami differential of $f_\mu$ is equal to $\tilde { \mu } $, and 
	\[
		f_\mu (z) = \frac {1} {z + i} + o(1), \ z \to -i. 
	\]
The Schwarzian derivative of $f_\mu$ is defined on $\mathbb{H}^*$ as
	\[
		S(f_\mu |_{\mathbb{H}^*} ) = \left( \frac{f_\mu''} {f_\mu'} \right)' - \frac {1} {2}\left( \frac{f_\mu''} {f_\mu'} \right)^2. 
	\]
The mapping $\mathcal B: \text{Bel}(\Gamma) \ni \mu \mapsto S(f_\mu |_{\mathbb{H}^*} ) \in B(\Gamma)$ is well-defined and compatible with Teichm\"uller equivalence. In other words, it establishes a holomorphic map from $\text{Teich}(\Gamma)$ to $B(\Gamma)$. The map is a biholomorphic map onto its image, called the \textit{Bers embedding}. 
As a known fact, its image is a bounded domain in $B(\Gamma)$. 
Therefore, the boundary can be defined via the Bers embedding. 
This is called the \textit{Bers boundary} $\partial_{\text{Bers}} \Teich(\Gamma) : =  \partial \mathcal B ( \Teich(\Gamma) )$. 
Next, we summarise the holonomy homomorphism defined at each point in $B(\Gamma)$. 
For each a point $\varphi \in B(\Gamma)$, 
then there exists a unique, local univalent holomorphic map $W_\varphi: \mathbb H ^\ast \to \hat{\mathbb C}$ such that 
	\[
		S(W_\varphi ) = \varphi, \ \ \ W_\varphi(z) = \frac {1} {z + i} + o(1), \text{near} \ z = -i. 
	\] 
Then, to every an element $\gamma \in \Gamma$, there exists a unique M\"obius transformation $g$ such that 
	\[
		W_\varphi \circ \gamma = g \circ W_\varphi. 
	\]
We denote the mapping $\gamma \mapsto g$ by $\chi_\varphi$, called a monodromy isomorphism. If $W_\varphi$ 
is an univalent, $\chi_\varphi(\Gamma)$ is also discrete. From Nehari's theorem (\cite[\S VI Lemma 3]{QC}),  
if $\varphi $ is contained in $\Cl(\Teich(\Gamma)) : = \Teich(\Gamma) \cup \partial_{\text{Bers}} \Teich(\Gamma)$
a function $W_\varphi$ is a univalent function.
Then, we classify $\varphi$ according to the properties of $\chi_\varphi(\Gamma)$. 

\begin{de} [Classification of elements in $B(\Gamma)$] \label{cal.kleiniangroup}

	Let a function $\varphi$ in $B(\Gamma)$. The group $\chi_\varphi ( \Gamma)$ or the function $\varphi$ is called
	
		\begin{itemize}
		
			\item a \textit{quasiFuchsian group} 
			if the function $W_\varphi$ can be extended to a homeomorphic map on $\hat{\mathbb C}$, 
			
			\item a \textit{cusp} if there exists a hyperbolic element 
			$\gamma \in \Gamma$ such that $\chi_\varphi(\gamma)$ is a parabolic element, 
			
			\item a \textit{totally degenerate group}
			if $W_\varphi(\mathbb H^\ast) \subset \hat{\mathbb C}$ is an open and dense subset. 
			
		\end{itemize}
		
	In this paper, we call the group  $\chi_\varphi ( \Gamma)$ or the function $\varphi$ 
	a \textit{topological quasi-Fuchsian b group}
	if $\varphi \in \partial_{\text{Bers}} \Teich(\Gamma)$
	and $\varphi$ is a quasiFuchsian group. In particular, when the function $W_\varphi$ satisfies a David condition 
	(see Definition \ref{def:davidcondition}), we call $\chi_\varphi ( \Gamma)$ or the function $\varphi$  
	a \textit{\SQF}. 
\end{de} 

Bers \cite{B} and Maskit \cite{Mas} proved that if a finitely generated Fuchsian group $\Gamma$ is of the first kind,
topological quasiFuchsian b-group does not exist in the Bers boundary. 
The following are known. 

	\begin{theo} [\cite{B} and \cite{Mas}] \label{resultofBersMaskit}
	 
	 	Let $R$ be a Riemann surface whose Fuchsian group is finitely generated of the first kind. 
		If a function $\psi \in \partial_{Bers} \Teich (S)$ is not a cusp, it is a totally degenerate group. 
		Moreover, if a function  $\varphi \in \Cl(\Teich(\Gamma)$ 
		is a quasiFucsian group, then $\varphi$ is in $\Teich(\Gamma)$.

	\end{theo}

The others obtained results of Tanigawa \cite[Proposition 4.4]{Th} and Yamamoto \cite[Theorem 3.1]{Yamamoto} 
after the Theorem \ref{resultofBersMaskit} and the comparison and relations with the main results 
are summarized in Section \ref{comp}.

\subsection{David map}

This section summarizes the David map, characterized as a solution to the degenerate Beltrami equation. 
In particular, we discuss the results studied by David \cite{D}. 
Since the quasiconformal map is characterized as a solution of the Beltrami equation, 
the David map can be regarded as a degenerate version of the quasiconformal map. 

	\begin{de} [ ] \label{def:davidcondition}
	 
	 	Let $\mu$ be a measurable function on an open set $U \subset \mathbb C$
		and $| \mu (z) | < 1$ for almost all $z \in U$ with respect to the Lebesgue measure. We say that 
		$\mu$ satisfies \textit{a David condition}, if there exist $\alpha > 0$, $C > 0$ and $\varepsilon_0 > 0$
		such that for each $\varepsilon \geq \varepsilon_0$
			\[
				\left| \{ z \in U \mid | \mu ( z ) | > 1 - \varepsilon \} \right|_{\Leb}
				\leq C \exp \left( - \frac {\alpha} {\varepsilon} \right), 
			\]
		where $| \cdot |_{\Leb}$ is the Lebesgue measure on the plane. 
		Moreover, if $K_\mu(z) : = \frac {1 + | \mu |(z)} {1 - |\mu|(z)}$ 
		satisfies
		$e^{K_\mu} \in L^p_\loc$ for some $p \in ( 0, \infty )$, then $\mu$ 
		is said to have \textit{exponentially $L^p_\loc$ integrable distortion}. 
	\end{de}
	
	\begin{Rem}
	 
	 	Suppose that $\supp ( \mu ) \subset \Cl ( \mathbb D )$ and $K_\mu \in L^p ( \mathbb D )$, 
		then $\mu$ satisfies the David condition by Chebychev's inequality, 
			\begin{align*}
				\left| \{ | \mu | > 1 - \varepsilon \} \right|_{\Leb}
				& = \left|\left\{ K_\mu + 1 \geq \frac {2} {\varepsilon} \right\} \right|_{\Leb} 
				=  \left| \{ \exp ( K(z) + 1 ) \geq e^{2 / \varepsilon} \} \right|_{\Leb} \\
				& \leq e^{-2 / \varepsilon} \int_{\mathbb D}   \exp ( p ( K(z) + 1 ) ). 
			\end{align*}
	\end{Rem}
	
	\begin{theo} [Generalized Riemann mapping theorem, {\cite{D}}] 
	 
	 	Let $\mu$ be a measurable function on $\mathbb C$ satisfies the David condition, 
		then the Beltrami equation $f_{\bar z} = \mu ( z ) f_z(z)$ 
		has a homeomorphism solution $f_\mu \in W^{1,1}_{\loc}(\mathbb C)$ 
		which is unique up to postcomposition with a conformal map.

	\end{theo}
	
	\begin{de} [ ] 
	 
	 	Let $\mu \in L^\infty ( \mathbb C )$ satisfy $| \mu ( z ) | < 1$ for allmost everywhere $z \in \mathbb C$. 
		A homeomorphism $f: \mathbb C \to \mathbb C$ 
		is called \textit{the principal solution} to the Beltrami equation with coefficient $\mu$
		if the map $f$ satisfies the following conditions: 
		
			\begin{itemize}
				\item $f$ satisfies the partial differential equation $f_z = \mu f_z \ \text {a.e.}$, 
				\item there exists a finite set $E \subset \mathbb C$ such that 
				$f \in W^{1,1}_{\loc} ( \mathbb C \setminus E )$. 
				\item around the point at infinity, $f$ has the Laurent expansion $f(z) = z + \frac {a_1} {z} + \cdots$. 
			\end{itemize}
	 
	\end{de}
	
	\begin{theo}  [{\cite[Theorem 20.4.7]{AIM}}] \label{principle sol}

	 	Let $\mu$ be a measureable function on $\mathbb C$ with $| \mu | < 1 \text{a.e.} $. 
		If there exist $p > 1$ and a measurable function $K$ such that
		$e^K \in L^p ( \mathbb D )$ and $| \mu ( z ) | \leq \frac {K(z) - 1}{ K ( z ) + 1} \chi_{\mathbb D}$, 
		where $\chi_{\mathbb D}$ is the characteristic function of $\mathbb D$, 
		the Beltrami equation $f_{\bar z} = \mu ( z ) f_z(z)$ 
		admits a unique principal solution $f \in W^{1,2}_{\loc} ( \mathbb C )$.

	\end{theo}

	\begin{prop} \label {loc unif convergence}

	 	Suppose that $\mu$ and $(\mu_n)$ satisfy 
		$\exp(K_\mu) \in L^p_\loc$ for some $p \in ( 0, \infty )$ and $e^{K_{\mu_n}} \in L^{p_n}_\loc$ 
		for some $p_n \in ( 0, \infty )$, respectively. 
		If $( \mu_n )$ satisfies $\| K_{\mu_n} \|_{L^1(\mathbb D)} \to \pi$ (as $n \to \infty)$,  
		the sequence of the normal solution $f_{\mu_n}$
		converges uniformly to the identity map on $\hat{\mathbb C}$. 		
	
	\end{prop}
	
	\begin{proof}
	 
	 	Denote the inverse map of $f_{\mu_n}$ by $g_n$. Then $g_n$ has the following modulus of continuity:		
		for $a, b \in \mathbb C$, 
			\[
				|g_n (a) - g_n(b)| \leq 16 \pi^2 \frac {1 + |a|^2 + |b|^2 } {\log\left( e + \frac {1} {|a - b |} \right)}
				\| K_{\mu_n} \|_{L^1(\mathbb D)}, 
			\] 
		which is in \cite[Lemma 20.2.3]{AIM}. Therefore, the sequence $(g_n)$ is a normal family. 
		We obtain a subsequence of $(g_n)$, converging uniformly in $\hat{\mathbb C}$ to a mapping $g$. 
		
		Moreover, the sequence $(g_n)$ also satisfies the inequality: 
			\[
				\int_{\mathbb C} | (g_n)_z - 1| ^2 + | (g_n)_{\bar z} |^2 \leq \| K_n \|_{L^1(\mathbb D)}, 
			\] 
		thus $g$ is in $W_{\mathrm{loc}}^{1,2}(\mathbb C)$. 
		This same argument and technique are also written in \cite[Theorem 20.2.1, Lemma 20.2.2]{AIM}. 
		
		To prove Proposition \ref{loc unif convergence}, 
		we only have to prove that $g$ is conformal on $\hat{\mathbb C}$. 
		First of all, it is obvious that $g$ is conformal on $g^{-1}( \hat{\mathbb C} \setminus \Cl(\mathbb D) )$, 
		because $g_n$ are conformal on $g_n^{-1}( \hat{\mathbb C} \setminus \Cl(\mathbb D) )$ and
		the subsequence of $(g_n)$ is uniformly converging to $g$. Let $\mathbb D_2$ be the disk of 
		radius $2$ at the center $0$. 
		Next, from a standard argument, 
		we get 
			\begin{equation} \label{prop1.1}
				\int_{g_n^{-1}(\mathbb D)} | (g_n)_{\bar z} | ^2 \leq
				\int_{g_n^{-1}(\mathbb D)} ( | (g_n)_{\bar z} | + | (g_n)_z | ) ^2
				\leq \| K_{\mu_n} \|_{L^1(\mathbb D)}. 
			\end{equation}
		In addition, we have
			\[
				\int_{g_n^{-1}(\mathbb D_2)} | (g_n)_z | ^2 \geq 
				\int_{g_n^{-1}(\mathbb D_2)}  | (g_n)_z | ^2 
				\left( 1 - \left| \frac {(g_n)_{\bar z}}{(g_n)_z} \right|^2 \right) = 4 \pi. 
			\] 
		Then, extending $K_{\mu_n}$ by 1 over the set $\{ 1\leq |z| \leq 2 \}$,   
		we obtain the inequality from (\ref{prop1.1}): 
			\[
				\int_{g_n^{-1}(\mathbb D_2)} | (g_n)_{\bar z} | ^2  \leq
				\| K_n \|_{L^1(\mathbb D_2)} - \int_{g_n^{-1}(\mathbb D_2)} | (g_n)_z | ^2
				 \leq \| K_{\mu_n} \|_{L^1(\mathbb D_2)} - 4 \pi. 
			\]
		Therefore, let $n \to \infty$, it follows that 
			\[
				\int_{g^{-1}(\mathbb D_2)} | g_{\bar z} | ^2 = 0, 
			\]
		because $\| K_{\mu_n} \|_{L^1(\mathbb D_2)} \to 4 \pi$ as $n \to \infty$ 
		from the assumption in Proposition \ref {loc unif convergence}. 
		(Notice $K_n \equiv 1$ on the set $\{ 1\leq |z| \leq 2 \}$.) From Wely's lemma, 
		$g$ is holomorphic on $g^{-1}(\mathbb D_2)$. 
		
		Finally, since $g$ is holomorphic on $\hat{\mathbb C}$ and 
		is univalent in the neighborhood of $z = \infty$, $g$ is conformal on $\hat {\mathbb C }$, 
		 using the fact that the degree of the holomorphic function is constant. 
		 Moreover, because the Laurent expansion of $g_n$ is
		 $g_n(z) = z + a_{1, n} / z + \cdots$, we can prove $g \equiv \id$ on $\hat{\mathbb C}$. 
		 That is, the sequence $( f_{\mu_n} )$ has a subsequence that uniformly converges to the identity map on 
		 $\hat{\mathbb C}$. 
	\end{proof}

\section{Deformation of a family of cylinders}\label{sectiondefcylinderfam}

\subsection{A result of McMullen}

Let $R$ be a Riemann surface with $\short ( R ) > 0$, where $\short ( R )$ is the infimum of the length of the  closed geodesic on $R$. 

\begin{theo} [{\cite[Theorem 1.2]{McM}}] \label{McMullen result}

Let $[ \mu ]$ be a point in $\Teich ( R )$, and $\nu$ be a unit-norm Beltrami differential in the tangent space $T_{[ \mu ]} \Teich ( R )$ at $[\mu]$. 
If the support of $\nu$ is contained in the part of $f^\mu(R)$ of injectivity radius less than $L < 1/2$, where the injectivity radius is measured using the hyperbolic metric of $R$. 
Then, the image of $\nu$ under the derivative of Bers' embedding $\mathcal B: \Teich(\Gamma) \to B(\Gamma)$ has norm at most $C ( L \log 1/ L ) ^2$, where the constant $C$ depends only on $\short(R)$. 	 

\end{theo}
	
\begin{Rem}
In the original paper \cite{McM}, it seems that $R$ is assumed to be of analytically finite type, but it is easy to see that his proof works without the assumption that $R$ is to be of analytically finite type. 
\end{Rem}

\subsection{Variation of injectivity radius in stretch deformation} \label {deform cyl}

We denote by $\St(H) : = \{ \zeta \in \mathbb C \mid | \IM \zeta | < H \}$, $C(H) : = \St(H) / \langle x \mapsto x + 2 \pi \rangle$, where $H > 0$. This cylinder is biholomorphic to $\{ z \in \mathbb C \mid 1/e^H < |z| <e^H \}$. In particular, the modulus is $H / \pi$.
The simple closed curve $[0, 2\pi] \ni x \mapsto x $ in $C(H)$ is called the core curve.  The core curve of a region $\Omega$ biholomorphic to $C(H)$ is defined as the image of the conformal map from $C(H)$ to $\Omega$. 
Let $H > 0$ and $0 < a < 1$. In the following, we consider a quasiconformal deformation of $\St(H)$ in which the subdomain $\St(a H)$ is stretched by a factor of two and $\St(H) \setminus \St(a H)$ is shifted in parallel. 
This quasiconformal deformation induces a deformation of $C(H)$. 
That is, the quasiconformal map and homotopy we are dealing with are as follows: 
	\begin{equation} \label{def: quasiconformal deformation cylinder}
		\psi_{H}  : x + iy \mapsto
		\begin{cases}
			x + iy + a H i & a H \leq y \leq H \\
			x + 2 iy & | y | \leq a H \\
			x + iy - a H i & - H \leq y \leq - a H
		\end{cases} \in C \left( (1 + a ) H \right), 
	\end{equation}
For each $t \in [0, 1 ]$, we define 
	\begin{equation}  \label{def: quasiconformal homotopy cylinder}
		\psi_{H, t }  : x + iy \mapsto
		\begin{cases}
			x + iy + a t H i & a H \leq y \leq H \\
			x + (1 + t ) iy & | y | \leq a H \\
			x + iy - a t H i & - H \leq y \leq - a H 
		\end{cases} \in C \left( ( 1 + at ) H \right) .
	\end{equation}
We denote by $\bel (\psi_{H, t})$ the beltrami differntial $( \psi_{H, t})_{ \bar z} / (\psi_{H, t})_{z}$. Then, 
	\begin{equation} \label{the bel differntial of stretch def}
		\bel ( \psi_{H, t} ) ( x + iy ) = 
		\begin{cases}
			- \frac {t} {2 + t} & | y | \leq a H \\
			0  & \text{the other}. 
		\end{cases}
	\end{equation}
and
	\begin{equation} \label{the infinitesimal bel differntial of stretch def}
		\left. \frac{d}{dh} \bel ( \psi_{H, t+h} \circ \psi_{H, t}^{-1} ) \right|_{h = 0} ( x + iy ) = 
		\begin{cases}
			- \frac {1} {2 + 2t} & | y | \leq (1+t) a H \\
			0  & \text{the other}. 
		\end{cases}
	\end{equation}
Note that this Beltrami coefficient and infinitesimal Beltrami coefficient are independent of the height $H$ of the cylinder and the constant $a$. 

Denote by $\rho^t$ the hyperbolic metric on $C(( 1 + a t ) H)$ and by $r_{\rho^t} ( p )$ the injectivity radius of $C({( 1 + a t ) H})$ at $p \in C({( 1 + a t ) H})$. 

\begin{prop} \label{prop: esitimates of inj radius}

 	Taking $a = 1/2$, if $H > 2 \sqrt 2 \pi^2$,  
		\begin{equation}\label{prop-eq cal inj rad}
			\sup_{p \in C((1+t)a H)} r_{\rho^t} ( p ) \leq 1/2.
		\end{equation}
	In particular, 
		\begin{equation}\label{prop-eq rat inj rad}
			\frac{\sup_{p \in C(3H/4)} r_{\rho^1} ( p )} {\sup_{p \in C(H/2)} r_{\rho^0} ( p )} \leq \frac{2 \sqrt 2} {3}  ( < 1 ). 
		\end{equation}
	Similarly, 
	we denote the lengths $l_0, l_1$ of the shortest simple closed geodesics in 
	$C(H)$ and $C( 3H/2 )$, respectively, then we get 
		\begin{equation}\label{prop-eq rat len corecurve}
			\frac{l_1} {l_0} = \frac{2}{3}. 
		\end{equation}

\end{prop}

\begin{proof}

Note that the hyperbolic metric on $C(( 1 + a t ) H)$ is 
	\[
		 \frac {\pi} {2 ( 1 + at ) H \cos \left( \frac {\pi} {2(1+ at ) H} \IM z \right)} |dz|. 
	\]

First,  consideing the closed curve  $g:  [ -\pi, \pi ] \ni x \mapsto x +  (1 + t)aHi $, we get 
	\[
		\sup_{p \in C((1+t)a H) } r_{\rho^t} ( p ) \leq \length_{\rho^t} ( g ) = 
		\frac {\pi^2} {(1 + at) H \cos \left( \frac {a(1+t)} {2(1 + at)} \pi \right)}. 
	\]
Since we take $a = 1/2$, 
	\[
		\frac {\pi^2} {(1 + at) \cos \left( \frac {a(1+t)} {2(1 + at)} \pi \right)} = 
		\frac {2 \pi^2} {(2 + t) \cos \left( \frac {1+t} {4 + 2t} \pi \right)} 
		\leq \sqrt 2 \pi^2, 
	\]
for all $t \in [0, 1 ]$. Therefore, if $H > 2 \sqrt 2 \pi^2$, then
	\[
		\sup_{p \in C((1+t)a H) } r_{\rho^t} ( p ) \leq \frac{1} {2}. 
	\]

Next, we estimate $\sup_{p \in C(aH) } r_{\rho^0} ( p )$. 
Since the geodesic through $0 \in C({a H})$  with respect to $\rho^0$ is $g_0 : [ -\pi, \pi ] \ni x \mapsto x \in C({a H})$ and its length is thde shortest geodesic in  $C(aH)$, we get
	\begin{eqnarray}\label{eq: estimate for inj rad}
		\frac {\pi^2} {H}= \length_{\rho^0} ( g_0 ) \leq \sup_{p \in C({a H}) } r_{\rho^0} ( p ). 
	\end{eqnarray}
Therefore, 
	\begin{equation}
		\frac{\sup_{p \in C(3H/4)} r_{\rho^1} ( p )} {\sup_{p \in C(H/2)} r_{\rho^0} ( p )}
		\leq \frac {2} {3  \cos \left( \frac {\pi} {4} \right)} = \frac{2\sqrt 2} {3}.
	\end{equation}
	
Finally, the shortest closed geodesics of $C(H)$ and $C(3H/2)$ are $[ -\pi, \pi ] \ni x \mapsto x$, so their lengths are $\pi^2 / 3$ and $2 \pi^2 / 3 H$ respectively.  Then, the ratio is $2 / 3$.

\end{proof}
 Summing up the above calculations, we find the following facts by Theorem \ref {McMullen result}: 
 Let $a = 1/2$ and $H > 2 \sqrt 2 \pi^2$. Then, we get 
	\begin{equation} \label{fund.estimate.McM.}
	\begin{split}
		& \| \mathcal B ( [ \bel(\psi_{H, 0}) ] ) -  \mathcal B ( [ \bel(\psi_{H, 1} ) ] ) \|_{B(C(H))}
		\leq
		\int_{0}^1 \left\| \frac {d} {dt} \mathcal B ( [ \bel ( \psi_{H, t } ) ] ) \right\|_{B(C(H))} \dt \\
		&\leq   
		\int_{0}^1 \left\| d_{[ \bel ( \psi_{H, t } ) ] }\mathcal B \right\|_{B(C(H))} 
		\left| \left. \frac{d}{dh} \bel ( \psi_{H, t+h} \circ \psi_{H, t}^{-1} ) \right|_{h = 0}  \right|  \dt \\
		&\leq C \left( L_0 \log \frac {1} {L_0} \right)^2 \int_0^1 \frac {1} {2t + 2} \dt, 
	\end{split}
	\end{equation}
where $L_0 : = \sup_{p \in C(a H) } r_{\rho^0} ( p )$.

\subsection{The deformation of a cylinder on $R$ } \label{subsec: def cylinder on R}

\begin{figure}[h]
\centering
\includegraphics[scale=0.3]{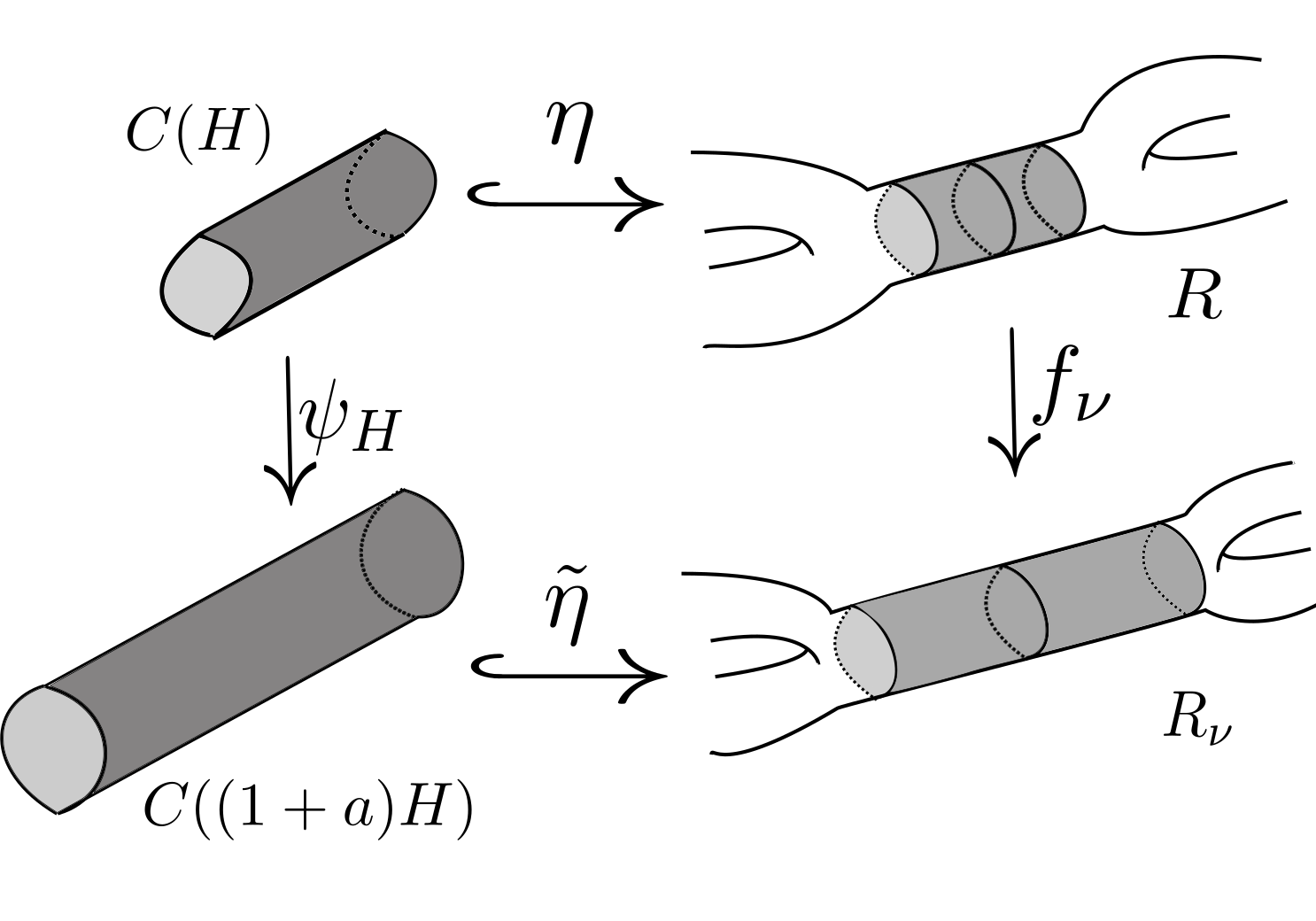}
\caption{}
\label{fig defofcly}
\end{figure}
We will verify that the evaluation of \eqref {fund.estimate.McM.} can also be applied to the appropriate cylinder region of the Riemann surface $R$. 
On Riemann surfaces, the following theorem is a way to pick out cylinders from simple closed curves. 

\begin{theo} [The collar Lemma, {\cite[Theorem 3.8.3]{H}}] \label{collar}

 	Let $R$ be a Riemann surface. 
	For an essential simple closed curve $\gamma$ in $R$, we write the geodesic whcih is homotopic to $\gamma$ 
	as $\gamma^\ast$. 
	For positive $\delta > 0$, we define the domain $A_\delta(\gamma) : = \{ x \in R \mid d(x, \gamma^\ast) < \delta\}$.
	Then, the following holds: 
	Considering a (even infinite) set of essential simple closed curves $\{ \gamma_1, \gamma_2, \cdots \}$, 
	if each geodesic is disjoint from the other, 
	then the domains $A_{h(l_i)} (\gamma_i)$ are biholomorphic to cylinders  and are disjoint, 
	where $l_i$ is the length of the geodesic $\gamma_i^\ast$ and 
	$h(x) = \frac{1}{2} \log \frac {cosh (x/2) + 1}{{cosh (x/2)} - 1}$.

\end{theo}

\begin{Rem}
 
 	The height of $A_{h(l_i)} (\gamma_i)$ is  $h(l_i)$. The function $h$ is monotone-decreasing such that 
	$h(l) \to \infty$ as $l \to 0$. 
 
\end{Rem}

Let $\gamma^\ast$ be an essential simple closed geodesic in $R$. 
The height of the cylinder must be large enough for our purpose. It means that the length $l$ of $\gamma^\ast$ must be small enough. 
We assume that there exists an essential simple closed geodesic $\gamma^\ast$ such that the length $l$ of $\gamma^\ast$
satisfies $h(l) > 2 \sqrt 2 \pi^2$.  
Then, the injectivity radius of $A_{h(l)/2} (\gamma^\ast)$ with respect to $\rho_{A, l}$ is less than $1/2$. 
We define $\Cyl : = A_{h(l)} (\gamma^\ast)$ and $\rho^0 : = \rho_{A, l}$. 
Moreover, we take a standard holomorphic embedding $\eta: C(H) \to \Cyl$.
See Figure \ref{fig defofcly}. 
Define the quasiconformal deformation of $R$ determined by $\Cyl$, $H$, and $\eta$. 
Using this map $\eta$, we can define the quasiconformal deformation constructed by (\ref{def: quasiconformal deformation cylinder}) in the previous subsection on $R$. More precisely, the quasiconformal deformation can be determined by the following Beltrami coefficients \eqref{def: quasiconformal deformation cylinder} and the homotopy from 
\eqref{def: quasiconformal homotopy cylinder}: 
	\[
		\nu: = \begin{cases}
			\eta_{\ast}(\bel(\psi_H)) & \text{on} \ \Cyl \\
			0 & \text{on} \ R \setminus \Cyl
		\end{cases}, \ \ \ 
		\nu_t : = \begin{cases}
			\eta_{\ast}(\bel(\psi_{H, t })) & \text{on} \ \Cyl \\
			0 & \text{on} \ R \setminus \Cyl
		\end{cases}. 
	\]
Note that this quasiconformal deformation depends not only on $\Cyl$ but also on $\eta: C(H) \to \Cyl$. Let us write $f_\nu$ and $f_{\nu_t}$ for the quasiconformal maps determined from $\nu$ and $\nu_t$, respectively. We also write $\rho_{R_{\nu_t}}$ and $\rho_{R_\nu}$ as hyperbolic measures of $R_\nu : = f_\nu(R)$ and $R_{\nu_t} : = f_{\nu_t} (R)$, respectively. 

\begin{claim}\label{standard embedding}

	There exists a holomorphic embedding $\tilde \eta: C(3H/2) \to f_\nu (\Cyl)$. 
	Similarly, for each $t \in [ 0 , 1 ]$, there exists a holomorphic embedding $\tilde \eta_t: C( H + t H/2 ) \to f_{\nu_t} (\Cyl)$. 

\end{claim}

\begin{proof}

 	We can define the map $\tilde \eta: C(H + t H/2 ) \to f_{\nu_t} (\Cyl)$ as follows: 
		\[
			\tilde \eta_t (z) : = f_{\nu_t} \circ \eta \circ (\psi_{H, t})^{-1} (z). 
		\] 
	In addition, this map is holomorphic can be seen from the definition of the map $f_{\nu_t}$.

\end{proof}

\begin{claim} \label{infinitesimal Bel differntial for single}

	\[
		\left|  \left. \frac {d} {dh} \bel( f_{\nu_{ t + h}} \circ f^{-1}_{\nu_t} ) \right|_{h = 0} \right| 
		= \left| \left. \frac{d}{dh} \bel ( \psi_{H, t+h} \circ \psi_{H, t}^{-1} ) \right|_{h = 0}  \right|. 
	\]
	In particular, the left side of this equation does not depend on the height $H$ of the cylinder. 

\end{claim}

\begin{proof}

 	From Claim \ref {standard embedding}, we get 
		\begin{equation}\label{cal. bel of qc}
		\begin{split}
			\bel( f_{\nu_{ t + h}} \circ f^{-1}_{\nu_t} ) 
			&= \bel (\tilde \eta_{t+h} \circ \psi_{H, t+h} \circ \eta ^{-1} \circ \eta \circ \psi_{H, t}^{-1} \circ \tilde \eta_t) \\
			&= \left( \frac{| \tilde \eta_t ' | } {\tilde \eta_t '} \right)^2 \bel ( \psi_{H, t+h} \circ \psi_{H, t}^{-1}) \circ \tilde \eta_t
		\end{split}
		\end{equation}
	on $f_{\nu_t}(\Cyl)$. Moreover, from  \eqref{the infinitesimal bel differntial of stretch def}, 
	the infinitesimal Beltrami differential does not depend on the height of the cylinder.

\end{proof}

\begin{claim} \label{compare surface with model}

	The injectivity radius of $\supp \left( \left. \frac {d} {dh} \bel( f_{\nu_{ t + h}} \circ f^{-1}_{\nu_t} ) \right|_{h = 0} \right)$ 
	is less than $1/2$. 

\end{claim}

\begin{proof}

 	First, from \eqref{cal. bel of qc}, we get
		\begin{align*}
			\supp\left( \left. \frac {d} {dh} \bel( f_{\nu_{ t + h}} \circ f^{-1}_{\nu_t} ) \right|_{h = 0} \right)
			& = 
			\tilde \eta_t \left( \supp \left( \left. \frac{d} {dh} \bel(\psi_{H, t + h} \circ \psi^{-1}_{H, t}) \right|_{h=0} \right) \right)\\
			& = \tilde \eta_t ( \psi_{H, t} (C (H/2) ) ) =  \tilde \eta_t ( C ((1+t)H / 2) ) .
		\end{align*}
	For each $x \in \tilde \eta_t ( C ((1+t)H / 2) ) $, from Schwarz--Pick's lemma, 
		\begin{equation} \label{inj rad est on model} 
			r_{\rho_{R_{\nu_t}}} ( x ) \leq r_{\rho^t} ( \tilde \eta_t^{-1} ( x ) ) 
			\leq \sup_{p \in ((1+t)H/2) } r_{\rho^t} ( p ) \leq \frac{1} {2}. 
		\end{equation}
	The last inequality is from \eqref {prop-eq cal inj rad} in Proposition \ref{prop: esitimates of inj radius}, since $h(l) > 2 \sqrt 2 \pi^2$. 

\end{proof}
Of course, we get the same evaluation as \eqref{fund.estimate.McM.}. 
	\begin{equation} \label{fund.estimate.McM for surface}
	\begin{split}
		\| \mathcal B ( \nu_0 ) -  \mathcal B ( \nu_1 ) \|_{B(R)} & \leq
		\int_{0}^1 \left\| \frac {d} {dt} \mathcal B ( [\nu_t] ) \right\|_{B(R)} \dt \\
		& \leq   
		\int_{0}^1 \left\| d_{[\nu_t]}\mathcal B \right\|_{B(R)} 
		\left|  \left. \frac {d} {dh} \bel( f_{\nu_{ t + h}} \circ f^{-1}_{\nu_t} ) \right|_{h = 0} \right|  \dt \\
		&\leq C \left( L_0 \log \frac {1} {L_0} \right)^2 \int_0^1 \frac{1} {2+2t}  \dt, 
	\end{split}
	\end{equation}
where $L_0 : = \sup_{p \in \eta(C({H/2})) } r_{\rho_R} ( p )$, and the constant $C$ comes from Theorem \ref{McMullen result}.

\section{Construction of \SQF s} \label{sectionconstruction}

We would like to construct a topological quasi-Fuchsian b-group on the Bers boundary by repeatedly using the inequality \eqref {fund.estimate.McM for surface} obtained from McMullen's result. To do so, we 
\begin {itemize}
\item investigate whether a sequence $([\tau_n])_{n \in \mathbb N}$ given by a stretch deformation converges in the topology that defines Bers boundary in Subsection \ref {def of a family of cylinder}, 
\item construct the Beltrami differntial $\mu$ so that it satisfies the David condition in Subsection \ref {sqfconstruction},
\item show that the Schwarzian derivative $S(f_\mu)$ of the constructed David map $f_\mu$ is not an interior point in Subsection \ref {subsec: not interior}, and
\item see that it is the limit $\mu$ of the sequence of points that can be constructed by repeatedly applying the stretch deformation (slightly modifying the above sequence) in Subsection \ref {subsec: exists bers boundary}. 
\end{itemize}

\subsection{The deformation of a family of cylinders} \label {def of a family of cylinder}

Assume that $R$ has infinitely many essential closed geodesics $\{ \gamma_n^\ast\}_{n \geq 0}$ whose lengths are constant $l$ and satisfy $h(l) > 2 \sqrt 2 \pi^2$. We define $\{ \Cyl_n : = A_{h(l)} (\gamma_n^\ast ) \}_{n \geq 0}$. 
For each $n$, we fix a biholomorohic map $\eta_n$ between the n-th cylinder $\Cyl_n$ and a standard cylinder of height $H : = h(l)$. We will write the standard cylinder as $C(H)$. See Figure \ref{definioffamilycyl}. Denote by $U$ the union of the cylinders $\Cyl_n$. 

\begin{figure}[h]
\centering
\includegraphics[scale=0.4]{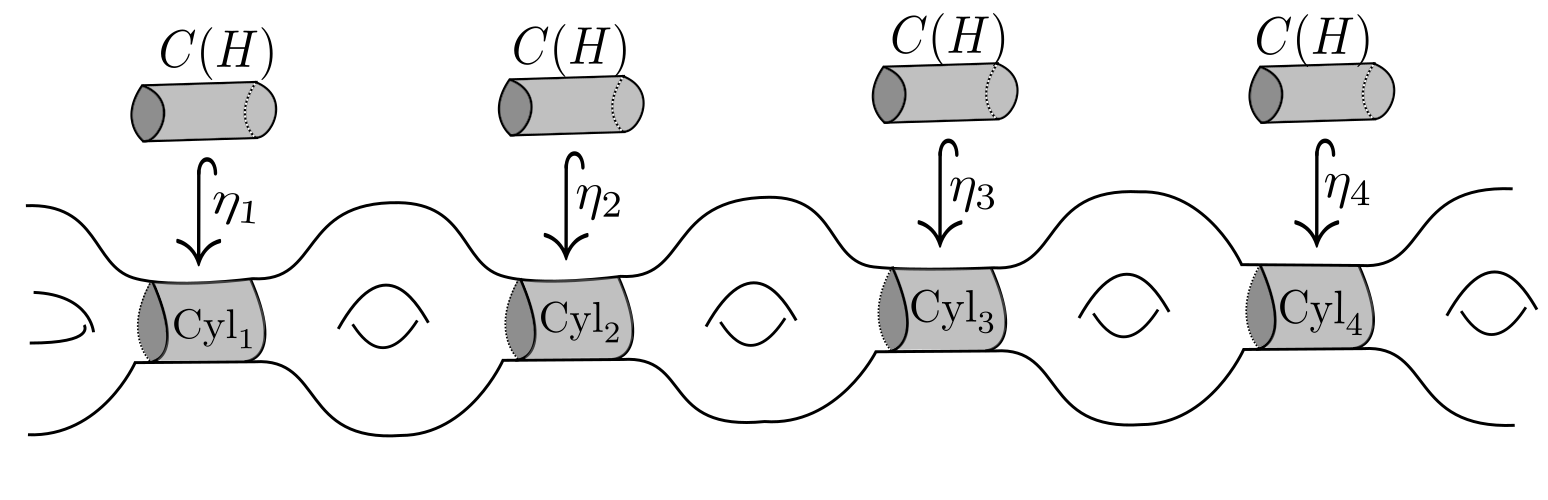}
\caption{}\label{definioffamilycyl}
\end{figure}

The sequence of quasiconformal deformation of $\{ R_j \}_{j \in \mathbb N}$ is defined inductively as follows: First, for each $\Cyl_n$, consider the quasiconformal deformation $f_t: R \to R_t$ determined by the Beltrami coefficient $\tau_t$ obtained by 
	\[
		\tau_1: = \begin{cases}
			(\eta_n)_{\ast}(\bel(\psi_{H})) & \text{on} \ \Cyl_n \\
			0 & \text{on} \ R \setminus U
		\end{cases}, \ \ \ 
		\tau_t : = \begin{cases}
			(\eta_n)_{\ast}(\bel(\psi_{H, t})) & \text{on} \ \Cyl_n \\
			0 & \text{on} \ R \setminus U. 
		\end{cases}. 
	\] 
Also, let $\Cyl_{(n,t)}: = f_t (\Cyl_n)$ for each $t \in [0, 1 ]$. As in Claim \ref{standard embedding}, the following holds. 

\begin{claim} \label{hol embedd}

	There exists a holomorphic embedding map $\eta_{n}^{(t)}$ from $\psi_{H, t}(C(H))$ to $\Cyl_{(n,t)}$. 

\end{claim}

\begin{proof}

 	We define 
		\[
			\eta^{(t)}_{n} = f_t \circ \eta_n \circ \psi_{H, t}^{-1}, 
		\]
	on $\psi_{H, t}(C(H))$. 
	By $\psi_{H, t}$ in (\ref{def: quasiconformal deformation cylinder}) and $f_t$, the function $\eta^{(t)}_{n}$ is holomorphic.

\end{proof}

In the same way as Claim \ref{compare surface with model}, the following can be proven. 

\begin{claim} \label{compare surface with model for infinite}

	The injectivity radius of each component of 
	$\supp \left( \left. \frac {d} {dh} \bel( f_{ t + h} \circ f^{-1}_t \right|_{h = 0} ) \right)$ is less than $1/2$. 

\end{claim}

\begin{proof}

 	First, 
		\begin{align*}
			\supp\left( \left. \frac {d} {dh} \bel( f_{t + h} \circ f^{-1}_t ) \right|_{h = 0} \right)
			& = 
			\bigcup_{n \geq 0} \eta_n ^{(t)} 
			\left( \supp \left( \left. \frac{d} {dt} \bel(\psi_{H, t + h} \circ \psi^{-1}_{H, t}) \right|_{h=0} \right) \right)\\
			& = \bigcup_{n \geq 0} \eta_n^{(t)} ( \psi_{H, t} (C (H/2))). 
		\end{align*}

	In the same manner as ( \ref {inj rad est on model} ), 
	for each $x \in \eta_n^{(t)} ( \psi_{H, t} (C (H/2)))$, from Schwarz--Pick's lemma, 
		\[
			r_{\rho_{R_{\tau_t}}} ( x ) \leq r_{\rho^t} ( (\eta_n^{(t)} )^{-1} ( x ) ) 
			\leq \sup_{p \in ((1+t)H/2) } r_{\rho^t} ( p ) \leq \frac{1} {2}. 
		\]
	The last inequality is from \eqref {prop-eq cal inj rad} in Proposition \ref{prop: esitimates of inj radius}, since $h(l) > 2 \sqrt 2 \pi^2$.

\end{proof}

\begin{figure}[h]
\centering
\includegraphics[width=100mm]{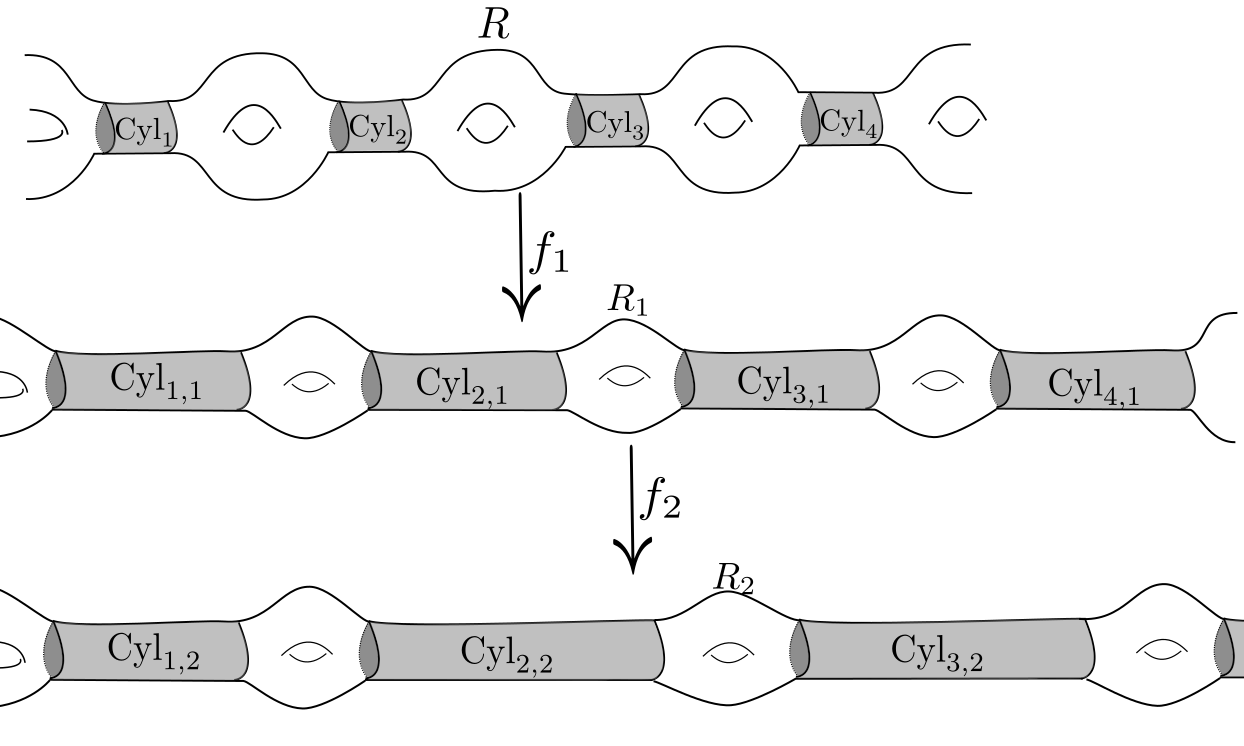}
\caption{}\label{defbydef}
\end{figure}

Next, we construct the $j$-th transformation. Consider a deformation of the standard cylinder $C(H)$ repeated $j$ times. In other words, using the notation from Subsection \ref{deform cyl}, we take the quasiconformal map 
	\begin{equation} \label{standard iterated deformation}
		\psi^{(j)} : = \psi_{(3/2)^j H} \circ \cdots \circ \psi_{3H/2} \circ \psi_{H}. 
	\end{equation}
In addition, we denote the injectivity radius of the support of $\bel (\psi_{(3/2)^j H})$ as $L'_j$. From \eqref{prop-eq rat inj rad} in Proposition \ref {prop: esitimates of inj radius}, for each $j \geq 0$, we get
	\begin{equation} \label{unif ratio}
		L'_j \leq \frac{2 \sqrt 2} {3} L'_{j+1}. 
	\end{equation}
Using the quasiconformal map $\psi^{(j)}$, we define 
	\begin{equation}\label{iterated deformation}
		\tau_j: = \begin{cases}
			(\eta_n)_{\ast}(\bel(\psi^{(j)})) & \text{on} \ \Cyl_n \ ( n \geq j ) \\
			0 & \text{on} \ R \setminus \bigcup_{n \geq j} \Cyl_n
		\end{cases}. 
	\end{equation}
We write the quasiconformal map determined by $\tau_j$ as $f_j: R \to R_j$, the image of $\Cyl_n$ by $f_j$ as $\Cyl_{(n , j )}$, and the injectivity radius of the support of $\tau_j$ as $L_j$ . See Figure \ref{defbydef}. 
Moreover, for each $t \in [ j-1 , j]$, we define
	\begin{equation} \label{model quasi def}
	\begin{split}
		\psi^{(t)} &: = \psi_{(3/2)^j H, t - (j-1)} \circ \psi_{(3/2)^{j-1}H}  \circ \cdots \circ \psi_{3H/2} \circ \psi_{H} \\
		& = \psi_{(3/2)^j H, t-(j -1)} \circ \psi^{(j-1)}
	\end{split}
	\end{equation}
and
	\begin{equation}\label{iterated deformation}
		\tau_t: = \begin{cases}
			(\eta_n)_{\ast}(\bel(\psi^{(t)})) & \text{on} \ \Cyl_n \ ( n \geq j ) \\
			0 & \text{on} \ R \setminus \bigcup_{n \geq j} \Cyl_n
		\end{cases}. 
	\end{equation}
Similarly, we define $R_t$, $f_t$, and $\Cyl_{(n , t)}$ for $\tau_t$. 
At this time, the same thing as Claim \ref{hol embedd} holds. 

\begin{claim}

	Let $j \geq 0$. For each $0 \leq t \leq j$ and $n \geq j$,  
	there exists a holomorphic embedding map $\eta^{(t)}_n$ from $\psi^{(t)}(C(H))$ to $\Cyl_{(n,t)}$. 

\end{claim}

\begin{proof}

 	We define 
		\[
			\eta^{(t)}_n = f_t \circ \eta_n \circ ( \psi^{(t)})^{-1}, 
		\]
	on $\psi^{(t)}(C(H))$.

\end{proof}

\begin{claim} \label{compare surface with model for infinite}

	The injectivity radius $L_{j-1}$ of each component of 
	$\supp \left( \left. \frac {d} {dh} \bel( f_{ t + h} \circ f^{-1}_t \right|_{h = 0} ) \right)$ 
	is less than $L'_{j-1}$, where $j$ is the nutural number so that $j - 1 \leq t \leq j$.  

\end{claim}

\begin{proof}

 	For each $t \in [0, \infty )$, we take $j \in \mathbb N$ with $j - 1 \leq t \leq j$. 
	Then, in the same way as Claim \ref{compare surface with model for infinite}, 
	we get
		
		\begin{align*}
			\supp\left( \left. \frac {d} {dh} \bel( f_{t + h} \circ f^{-1}_t ) \right|_{h = 0} \right)
			& = 
			\bigcup_{n \geq 0} \eta_n ^{(t)} 
			\left( 
			\supp \left( \left. \frac{d} {dh} \bel(\psi_{(3/2)^j H, t + h} \circ \psi^{-1}_{(3/2)^j H, t}) \right|_{h=0} \right) \right)\\
			& = \bigcup_{n \geq 0} \eta_n^{(t)} ( \psi_{(3/2)^j H, t} (C (H/2))), 
		\end{align*}
	and, for each $x \in \eta_n^{(t)} ( \psi_{(3/2)^j H, t} (C (H/2)))$, 
		\[
			r_{\rho_{R_{\tau_t}}} ( x ) \leq r_{\rho^t} ( (\eta_n^{(t)} )^{-1} ( x ) ) 
			\leq \sup_{p \in C \left( \frac {1 + t}{2} \left(\frac{3}{2}\right)^{j-1} H \right) } r_{\rho^t} ( p ) \leq L'_{j-1}, 
		\]
	from Schwarz--Pick's lemma. Therefore $L_{j-1} \leq L'_{j-1}$

\end{proof}

%Let $L_{j-1}$ be the injectivity radius of $\Cyl_{(m, j-1)}$, that is $L_{j-1} = \sup_{z \in \Cyl_{(m, j-1)}} r_{\rho_{j-1}} (z)$, where
%$\rho_{j-1}$ is the hyperbolic metric on $R_{j-1}$. 
%The heights of the deformed cylinders continue to increase without upper bound. That is, there exists a natural number $N$ such that the height of $\Cyl_{(n, j-1)}$ is bigger than $\max\{2 \sqrt 2 \pi^2, h(l_0) \}$, for each $n \geq j > N$. Note that 
%$N$ depends only on $\short (R)$. Furthermore, since the injectivity radius of $\Cyl_{(n, j-1)}$ is uniformly small, we can obtain $M$ as defined in subsection 3.3 independently of $j$ and $n$. 
%From (\ref{fund.estimate.McM for surface}), it follows that. for each $j > N$, 
%%	
To summarize the above, we can see that the sequence $(\mathcal B ( [\tau_j]))_{j \geq 0}$ converges with respect to the norm of the Banach space $B(R)$. First, for each $j \in \mathbb N$, we can see	
	\begin{align*}
		\| \mathcal B ( [ \tau_j ] ) - \mathcal B([\tau_{j-1}]) \|_{B ( R )}
		& \leq C \left( L_{j-1} \log \frac {1} {L_{j-1}} \right)^2 
		\int_{0}^1 \left|  \left.\frac {d} {dh} \bel( f_{ (j -1)+ t + h} \circ f^{-1}_{ (j -1) + t}) \right|_{h = 0} \right| dt \\
		& \leq C \left( L'_{j-1} \log \frac {1} {L'_{j-1}} \right)^2 
		\int_{0}^1 \frac {1} {2+2t} dt 
	\end{align*}
from Theorem \ref {McMullen result}, Claim \ref {infinitesimal Bel differntial for single}, and Claim \ref {compare surface with model for infinite}. 
Note that $C$ is constant depending only on $R$ and the function $s(x) = -x \log(x)$ is is monotonically decreasing in the range $x>1/e$. 
Since  $L'_j \leq 2\sqrt 2 / 3 L'_{j-1}$ and $L'_0 \leq 1/2$,  we obtain 
	\begin{equation} \label{ineq:sum}
	\begin{split}
		& \sum_{j \geq 0} \| \mathcal B ( [ \tau_j ] ) - \mathcal B([\tau_{j-1}]) \|_{B ( R )}  \leq
		\left( C \int_{0}^1 \frac {1} {2+2t} dt \right) \sum_{j \geq 0} \left( L'_{j-1} \log \frac {1} {L'_{j-1}} \right)^2 \\
		& \leq \left( C \int_{0}^1 \frac {1} {2+2t} dt \right)  
		\sum_{j \geq 0} \left( \frac{2 \sqrt 2} {3} \right)^{2j} \left( \log \left(\frac {3} {2 \sqrt 2}\right)^j \right)^{2} \\
		& \leq 
		\left( C \int_{0}^1 \frac {1} {2+2t} dt \right) \left( \log \left(\frac {3} {2 \sqrt 2}\right)\right)^{2}
		\sum_{j \geq 0}  j^2 \left( \frac{2 \sqrt 2} {3} \right)^{2j} < \infty
	\end{split}
	\end{equation}
Thus, the sequence $(\mathcal B ( [ \tau_j ] ) )_{j \in \mathbb N}$ converges. 
In the next subsection, we will discuss a modified version of this construction that provides the existence and uniqueness of solutions to the Beltrami equation corresponding to this limit, that is, the David condition in Definition \ref {def:davidcondition}.

\subsection{Construction method} \label{sqfconstruction}

As in subsection \ref{def of a family of cylinder}, we assume the existence of infinitely many essential closed geodesics $\{ \gamma_n^\ast\}_{n \geq 0}$ whose lengths are constant $l$ and satisfy $h(l) >2 \sqrt 2 \pi^2$. 
And, let $\Gamma$ be a Fuchsian model of $R$ acting $\mathbb D$ with an universal covering map $\pi : \mathbb D \to R$
and $\omega$ be a fundamental domain of $\Gamma$ in $\mathbb D$. 
For every $n$, we can find a doubly connected domain $\Cyl_n$ of $R$ containing $\gamma_n^\ast$ such that 
$\Cl ( \Cyl_n )$ is compact in $R$ and $n \neq m$ implies that $\Cl ( \Cyl_n ) \cap \Cl ( \Cyl_m ) = \emptyset$. 
For example, we can take the collar of $\gamma_n^\ast$. 
We denote by $T_m = \pi^{-1} (\Cyl_m) \cap \omega$. 
%From the assumption that $\short(R) > 0$, the lower bound of the height of each $C_n$ can be greater than $0$. 
%Note that, by  $\short(R) > 0$ , the heights of all $C_n$ are taken to be $H$. Moreover, by applying quasiconformal deformation to $R$, we can assume that the injectivity radius of $\pi(C_n)$ is larger than $\max \{ 2 \sqrt{2} \pi^2, H_1 \}$ for all $n \in \mathbb N$. 
Let $C(H)$ be a standard cylinder that is biholomorphic to $\Cyl_n$, and let the holomorphic $\eta_n: C{(H)} \to \Cyl_n$ be fixed.

Take a positive sequence $( p_j )_{j \in \mathbb N}$ such that $\sum_{j \in \mathbb N} p_j < \infty$. 

	\begin{lem} \label { Area is small enough }

	 	For any $\varepsilon > 0$, there exists an $N \in \mathbb N$ such that 
		$\left| \bigcup_{m > N} \Gamma T_m \right|_{\Leb} < \varepsilon$.

	\end{lem}
	
	\begin{proof}

	 	Since $\mathbb D = \Gamma \omega \supset \Gamma \left( \bigcup_{m \in \mathbb N} T_m \right)$,
		one easily sees the Lemma.

	\end{proof}

First, consider the following as the quasiconformal deformation of $C(H)$. In more detail, considering the quasiconformal deformation given by \eqref{model quasi def} in subsection \ref{def of a family of cylinder}, we define 
%(\ref{def: quasiconformal homotopy cylinder}) with $a=1/2$, denoted by $\psi^{1, n}$. The image of this conformal transformation, denoted by $\psi(C(H,n))$ is conformal to $C(3/2H)$ under the notation of subsection \ref {deform cyl}, that is $C(3/2H) = \{ \zeta \in \mathbb C \mid | \IM \zeta | < 3/2H \} /  \langle x \mapsto x + 2 \pi \rangle$. Next, taking $3/2H$ again as $H$, we consider the quasiconformal deformation given by (\ref{def: quasiconformal deformation cylinder}) and (\ref{def: quasiconformal homotopy cylinder}) again, denoted by $\psi^{2, n}$. Finally, the Beltrami coefficient is defined on $C_{n}$ as follows: 
%	
	\[
		\nu_j : = \bel ( \psi^{(j)}), 
	\]
for each $j \in \mathbb N$. Then, from Lemma \ref { Area is small enough }, there exists a natural number $M(j)$ such that 
	\[
		\sigma \left( \Gamma T_{M(j)} \right) \cdot e^{2 K_j }< p_j, 
	\]
where 
	\[
		K _j : = \frac {1 + \|( \eta_{M(j)} )_\ast \nu_j\|_{L^\infty}} {1 - \|( \eta_{M(j)} )_\ast \nu_j\|_{L^\infty}}.
	\] 
This is defined so that the product of the maximum dilatation of the quasiconformal deformation and the area of the support in the universal covering is less than $p_j$ when the deformation that repeats stretch deformation $j$ times is defined on the $M(j)$-th cylinder. 
Now, let $\mu$ be defined by 
	\begin{equation}\label{def:david beltrami coe}
	\begin{split}
		\tilde \mu ( z ) & : =
		\begin{cases}
			(\eta_{M(j)})_\ast \nu_j ( \pi(z) ) & \text{if} \ z \in T_{M(j)} \\
			0 & \text{if} \ z \in R \setminus \bigcup_{j \in \mathbb N} \Cl ( T_{M(j)} )
		\end{cases} \\
		& = \sum_{j \in \mathbb N} (\eta_{M(j)})_\ast \nu_j ( \pi(z) ) \chi_{T_{M(j)}}. 
	\end{split}
	\end{equation}
Here $\chi_{T_{M(j)}}$ is the characteristic function of $T_{M(j)}$. 
The function $\tilde \mu$ has exponentially $L^2$ integrable distortion. That is, $\tilde \mu$ satisfies the assumptions of Theorem \ref {principle sol}. Indeed, one can see from the following calculation: 
	\begin{equation} \label{exp. int. of david}
		\int_{\mathbb D} e^{2 K(\tilde \mu)} \leq 
		\sum_{j \in \mathbb N} \sigma \left( \Gamma T_{M(j)} \right) \cdot e^{ 2 K_j }
		< \sum_{j \in \mathbb N} p_j < \infty, 
	\end{equation}
where 
	\[
		K(\tilde \mu) ( z ) = \frac{1 + | \tilde \mu |} {1 - | \tilde \mu |} ( z ).
	\] 
Therefore, if we extend $\tilde \mu$ to be identically zero in $\hat {\mathbb C} \setminus \Cl ( \mathbb D )$,
the Beltrami equation for $\tilde \mu$ has the principal solution $f_{\tilde \mu}$. 
Moreover, pushing $\mu$ of $\tilde \mu$ on $R$ under the covering map $\mathbb D \to R$ induced the homeomorphism $f_\mu: R \to f_{\mu}(R)$. 
It is denoted by $R_\mu$ its image. 

\subsection{$R_\mu$ is not an interior point in $\Teich ( R )$} \label{subsec: not interior}

It follows immediately from the following theorem that it is not an interior point. 

	\begin{theo} [{\cite[Lemma 3.1]{W}}] \label{wolpart}

	 	Let $f: R \to R$ be a $K$-quasiconformal map 
		and $\gamma$ be a closed curve  on $R$. 
		Denote by $l$ the length of the hyperbolic geodesic in the free homotopy class of $\gamma$ 
		and by $l'$ the length of the hyperbolic geodesic in the free homotopy class of $f(\gamma)$. 
		Then
			\[
				l / K \leq l' \leq lK. 
			\]

	\end{theo}
	
If $R_\mu$ is an interior point in $\Teich ( R )$, a $K-$quasiconformal map exists from $R$ to $R_\mu$. 
Then, from Theorem \ref {wolpart}, it follows that the length of closed geodesics $f_\mu(\gamma_{M(j)}) ^\ast$ homotopic to $f_\mu(\gamma_{M(j)})$ in $R_\mu$ are greater than $\short(R) / K ( > 0 )$. 
On the other hand, for each $n \in \mathbb N$,
the length of hyperbolic geodesic of the shortest geodesic in $f_{\mu}(\Cyl_{M(n)})$ is less than 	
	\[
		\left(\frac{2} {3}\right)^{n -1} \frac {\pi^2} {H}, 
	\]
from \eqref{prop-eq rat len corecurve} in Proposition \ref {prop: esitimates of inj radius} and Shcwarz--Pick's lemma. 
It means that 
	\[
		\lim_{j \to \infty} \length_{\rho_{R_\mu}} (f_\mu(\gamma_{M(j)}) ^\ast) = 0. 
	\]
It is a contradiction.

\subsection{$R_\mu$ exists in the Bers boundary} \label{subsec: exists bers boundary}

To prove that $R_\mu$ exists in the Bers boundary, we show that the Schwarizian derivative $S(f_{\tilde \mu})$
of the principal solution $f_{\tilde \mu}$ is in the Bers boundary. By the definition of the Bers boundary,
it is sufficient to take a sequence $(\varphi_n)$ in $\mathcal {B}(\Teich ( R ) )$ conversing $\varphi_{\infty}$ in $B(R)$ and 
$\varphi$ coincide $S(f_{\tilde \mu})$. 
On the open set $U: = \bigcup_{j \in \mathbb N} \Cyl_{M ( j )}$, 
Applying $\{ \gamma_{M(j)}^\ast \}$ to the discussion in Subsection \ref {def of a family of cylinder}, 
we can take the deformation of a family of cylinders $( [\tau_j] )_{j}$ in Subsection \ref {def of a family of cylinder}. 
The sequence converges in $\Teich(R) \cup \partial_{\Bers} \Teich(R) $ from (\ref{ineq:sum}). 
Now, we denote $\mathcal {B} ([\tau_j] )$ by $\varphi_j \in B ( R )$, and the limit of $\mathcal {B} ([\tau_j] )$ by $\varphi_\infty$. 
Since $W_{\varphi_j}$ is convergent to $W_{\varphi_{\infty}}$ with respect to the norm in $B (R)$, the sequence is
local uniformly convergent on the lower half plane $\mathbb H^\ast$. 

Moreover, we can prove $\| K(\tau_j - \tilde \mu) \|_{L^1} \overset {j \to \infty}{\longrightarrow} 0$. 
Indeed, for each $j$, a function $\tau_j$ is a Beltrami coefficient of a quasiconformal deformation 
in which $\Cyl_{M(n)}$ is stretched twice $n$  times for each $1 \leq n \leq j$ 
and $\Cyl_{M(j)}$ is stretched twice $n$ times for each  $j > n$.
Indeed, from the definition, $\tau_j$ is coincide to $\tilde \mu$ in $ \bigcup_{1\leq n \leq j} \Cyl_{M (n)}$. 
Thus, the following inequality holds from (\ref {ineq:sum}) in Subsection \ref{def of a family of cylinder}: 
	\[
		\| K ( \tau_n - \tilde \mu ) \|_{L^1(\mathbb D)} \leq 
		\sum_{n \geq j} \sigma \left( \Gamma T_{M(n)} \right) \cdot e^{2 K_n}
		\leq \sum_{n \geq j} p_n \overset {j \to \infty}{\longrightarrow} 0. 
	\]
Therefore, the principle solution of $\tau_{M(n)}$
is locally uniform convergent to $f_{\tilde \mu}$, from Proposition \ref {loc unif convergence}. 
That is, $f_{\tilde \mu}$ and $W_{\varphi_\infty}$ are the limit of the sequence ($W_{\varphi_n}$). 
Hence $f_{\tilde \mu} = W_{\varphi_\infty}$. Therefore $S(f_{\tilde \mu}) = \varphi_\infty$. 
The following conclusions follow from the above: 

\begin{theo} [ ] \label{exi. of sqf}

 	Let $R$ be a Riemann Surface of analytically infinite type with $\short (R) > 0$ and 
	infinitely many homotopically independent essential closed geodesics 
	$\{ \gamma_n^\ast\}_{n \geq 0}$ whose lengths are bounded.  
	Then, there exists a \SQF \ in 
	the Bers boundary of $\Teich(R)$.

\end{theo}

\begin{proof}

 	From the assumption 
	that $R$ has infinitely many homotopically independent  essential closed geodesics 
	$\{ \gamma_n^\ast\}_{n \geq 0}$ whose lengths are bounded, 
	there exists a quasiconformal mapping $f: R \to f(R)$ 
	so that the set $\{ f(\gamma^\ast_n)^\ast \}$ of geodesics homotopic to $f(\gamma^\ast_n)$ satisfies 
	the conditions; whose lengths are constant $l$ and satisfy $h(l) > 2 \sqrt 2 \pi^2$. 
	In addition, $[f(R), f]$ is an element of $\Teich(R)$. 
	From the above, 
	we can apply the discussion so far to $f(R)$ and the set $\{ f(\gamma^\ast_n)^\ast \}$.

\end{proof}

\section{A holomorphic family of \SQF}

The main purpose of this section is to prove the following: 
	
	\begin{theo} [ ] \label {resultball}

	 	Let $R$ be a Riemann Surface of analytically infinite type with $\short (R) > 0$ and 
		infinitely many homotopically independent 
		essential closed geodesics $\{\gamma_n^\ast\}_{n \geq 0}$ whose lengths are bounded. 
		Then, 
		there exists an infinite-dimensional complex manifold in $\partial_{\Bers} \Teich (R)$ 
		which consists of \SQF.  
		That is, 
		there exists an unit ball $D$ of a Banach space 
		and a holomorphic map $F: D \to B(\Gamma)$, 
		where $F(D)$ is contained in the Bers boundary $\partial_{\Bers} \Teich (R)$, contains only \SQF, and 
		the dimension of the tangent space of the image is infinite.

	\end{theo}

\begin{proof}

To construct the infinite-dimensional complex manifold in Theorem \ref {resultball}, we use partial deformation on $R$. 
By applying the same arguments as in the proof of Theorem \ref {exi. of sqf}, we may assume that the lengths of the geodesics $\{\gamma_n^\ast\}_{n \geq 0}$ are constant. Also, we can assume there exists an infinite number of geodesics that do not intersect each other with the $\gamma_n^\ast$. For example, replace $\{\gamma_n^\ast\}_{n \geq 0}$ with $\{\gamma_{2n}^\ast\}_{n \geq 0}$. 
As each $n$, we consider the subcylinder $\widetilde {C_n}$ of the collar $\gamma_n^\ast$, 
whose center is $\gamma_n^\ast$ and height is $h(l) / 2$. Since the lengths are constant, their modulus are constant.
Let $R_2$ be the union of the tubular neighborhoods $\{\widetilde {C_n}\}$ of $\{ \gamma_n^\ast \}_{n \in \mathbb N}$. 
Appalling $\{\gamma_{n}^\ast\}_{n \geq 0}$ to the discussion in Subsection \ref {sqfconstruction}, 
we take the Beltrami differential $\mu$ which induces a David--Fuchsian b-group. From the construction of $\mu$, 
the support of $\mu$ is contained in $R_2$
Then, the complex structure of $R_1: = R \setminus R_2$ can still be deformed. Moreover, the deformation space of $R_1$
is still an infinite dimension space,
since there exists an infinite number of geodesics that do not intersect each other with the $\gamma_n^\ast$.
By deforming the complex structure on $R_1$, we construct a map such that Theorem \ref {resultball} is satisfied. 
	\begin{equation} \label{def: partial deforme}
		P^\mu: \Bel ( R_1 ) \ni \tau \mapsto S ( f_ {(\mu + \tau)} ) \in B( R ).
	\end{equation}
Here, $\mu+\tau$ is precisely $\mu \chi_{R_2}+\tau \chi_{R_1}$. These differentials satisfy the David condition. 
Indeed, let $\pi: \mathbb \mathbb D \to R$ be an universal cover, and $\tilde \mu$ and $\tilde \tau$ be lifts of $\mu$ and $\tau$, respectivily. Set
	
	\[
		K (z) : = \frac{1 + | \tilde \mu + \tilde \tau |} {1 - | \tilde \mu + \tilde \tau |}. 
	\]
We get 
	\begin{align*}
		\int_{\mathbb D} e^{2 K(z)} &\leq \left( \int_{\supp_(\tilde \mu)} + \int_{\supp( \tilde \tau)} \right) e^{2 K(z)} \\
		&\leq  \int_{\supp(\tilde \mu)} e^{2 K_\mu (z)} + \int_{\supp( \tilde \tau)} e^{2 K_\tau (z)}
		< \infty. 
	\end{align*}
However, since the supports of $\mu$ and $\tau$ are already included in $R_2$ and $R_1$, respectively, the characteristic functions are omitted hereafter. 
The image of $P^\mu$ is almost evident that the image of this map is contained in the Bers boundary. 
Indeed, we consider the sequence $\{ f_{\tau_j + \tau} \}$ with the sequence $f_{\tau_j}$ converging to $\mu$ in Subsection  \ref {sqfconstruction}. 
Therefore, to complete the proof of Theorem \ref {resultball}, we must consider that the image of this map is infinite-dimensional. 
To show that the image is infinite-dimensional, we decompose the above map \eqref {def: partial deforme}
into the composition of the two maps, 
 a translation $F^\mu$ and a partial deformation $P^0$. 
	\begin{align} 
		& P^0 : \Bel ( R_1 ) \ni \tau \mapsto S ( f_{ 0 + \tau } ) \in B(R), \label{partial deforme null} \\
		& F^\mu: ( B(R) \supset ) \ \Teich ( R ) \ni S(f_\tau) \mapsto S ( f_{\mu +\tau} ) \in B(R). \label{trans. by david}
	\end{align}
The computation of the dimension of the image of $F^\mu$ can be seen from the construction of $\mu$. In fact, From the construction \eqref{def:david beltrami coe}, the Beltrami differential $\mu$ satisfies the exponential integrability condition \eqref{exp. int. of david}, so 
 $f_{\mu + \tau}$ is in $W_{\loc}^{1,2}(R)$ from Theorem \ref {principle sol}. Therefore, the tangent map at the origin is represented by 
	\begin{equation} \label{eq:deriofbersembedd}
		d_0 F ^\mu [\nu] ( z ) = - \frac {6} {\pi} ( f_{\mu + 0} )_ z ^2  ( z ) 
		\int_{R} \frac {\nu ( \zeta ) } {(f_{\mu + 0} ( z ) - f_{\mu + 0} ( \zeta ) )^4}
		 ( f_{\mu + 0} ) _ z ^2  ( \zeta )
	\end{equation}
in the same manner as the calculation for Bers embedding. 
For detailed calculations, see Claim \ref{Bers embedding cal. for david} given later. 
Briefly stated, since $f_{\mu + 0}$ belongs to $W_{\loc}^{1,2}(\mathbb C )$, the coordinate transformation by the function $f_{\mu + 0} $ can be done by multiplying the Jacobian of $f_{\mu+0}$ by the function under integration. 
From equation \eqref{eq:deriofbersembedd}, 
we know that the image of $F^\mu$ is infinite-dimensional. 
Therefore, it is sufficient to show that the image of $P^0$ is infinite-dimensional. 
In Subsection \ref {tangentmap}, we discuss the partial deformation $P^0$ and compute the tangent map of partial deformation. 

\end{proof}

\begin{claim} \label{Bers embedding cal. for david}

 	Let $R$ be a Riemann surface, $R_2$ its sub-surface (but not necessarily connected), 
	and $R_1 : = R \setminus R_2$. 
	Let $\mu$ be the Beltrami differential whose support is contained in $R_2$, 
	and the Beltrami equation has the homeomorphic solution $f_\mu$. 
	Suppose the homeomorphic solution is an element of $W^{1,2}_{\loc}(R)$. Then
	the map defined in \eqref{trans. by david} 
	\[
		F^\mu: ( B(R) \supset ) \ \Teich ( R ) \ni S(f_\tau) \mapsto S ( f_{\mu +\tau} ) \in B(R)
	\]
	is differentiable, and its derivative at $0 \in \Bel(R_1)$ is the formula \eqref{eq:deriofbersembedd}:  
	\begin{equation*} 
		d_0 F ^\mu [\nu] ( z ) = - \frac {6} {\pi} ( f_{\mu + 0} )_ z ^2  ( z ) 
		\int_{R} \frac {\nu ( \zeta ) } {(f_{\mu + 0} ( z ) - f_{\mu + 0} ( \zeta ) )^4}
		 ( f_{\mu + 0} ) _ z ^2  ( \zeta ). 
	\end{equation*}

\end{claim}

\begin{proof}

	The following calculations are well known when $\mu$ is in the $\Bel(\Gamma)$ (cf. \cite[Theorem6.11]{IT}).
 
 	In order to simplify the symbols, we will use the notation $f_\mu$ instead of $f_{\mu+0}$. 
	For the sake of proof, everything is lifted by the universal covering 
	$\pi: \mathbb D \to R$ with $\Gamma$ in the Fuchsian group model. 
	The lifts of $\mu$, $f_\mu$ $R_1$ and $R_1$ by the projection $\pi$ 
	are written as $\tilde \mu$, $\tilde f_{\mu}$, $\tilde R_1$ and $\tilde R_2$, respectively.
	
	Let $\tilde \nu$ be an element in $L^\infty(\Gamma)$ and its support be contained in $\tilde R_1$.
	To show the claim, 
	it is sufficient to show that the limit 
		\[
			\lim_{t \to 0} \frac{1} {t} ( F^{\mu} (f_{t \tilde \nu}) - F^{\mu} (f_{0}) )
		\]
	exists and that the extremal limit is given by \eqref{trans. by david}, where $f_{t \tilde \nu}$ and $f_{0}$ are
	the normal solution of the Beltrami equation of $t \tilde \nu$ and $0$ ,respectively. 
	We define $\tilde\mu_t : = \tilde \mu + t \tilde \nu$. Since $\supp(\tilde\mu) \cap \supp(\tilde\nu) = \emptyset$, 
	$\tilde \mu_t$ satisfies the David condition, 
	so we denote the normal homeomorphism solution of the Beltrami equation of $\tilde \mu_t$ by $\tilde f_{t}$. 
	Moreover, set $g_t : = \tilde f_{t} \circ \tilde f_{\mu}^{-1}$, then we get
		\[
			\lambda_t = 
			\frac {( \tilde f_{\mu} )_z} {\overline { (\tilde f_{\mu} )_z }} 
			\frac{\tilde\mu_t - \tilde\mu} {1 - \bar{\tilde\mu} \tilde\mu_t} \circ \tilde f_{\mu}^{-1} =
			\begin{cases}
				\frac {( \tilde f_{\mu} )_z} {\overline { (\tilde f_{\mu} )_z }} \cdot
				\frac{t \tilde\nu}{1 - t\tilde\nu} \circ \tilde f_{\mu}^{-1} & \ \text{on} \  \tilde f_{\mu} ( \tilde R_1 ) \\
				0 & \ \text{on} \ \hat {\mathbb C} \setminus \tilde f_{\mu} ( \tilde R_1 ), 
			\end{cases}
		\]
	so, the norm of $\mu_t$ is less than $1$, for sufficietntly small. Therefore $g_t$ is quasiconformal with respect to
	the complex structure of $\tilde f_{\mu} ( \hat {\mathbb C } )$. 
	Apply $g_t$ to the standard argument (cf. \cite[Theorem6.10]{IT}).
	we can show that the following limit exists:  
		\[
			\lim_{t \to 0} \frac{1} {t} S ( g_t|_{\tilde f_{\mu}} ( \mathbb D^\ast ) ). 
		\]
	Moreover, we get the represent formula: 
		\[
			\lim_{t \to 0} \frac{1} {t} S ( g_t|_{\tilde f_{\mu}} ( \mathbb D^\ast ) ) ( \zeta)
			= \int_{\tilde f_{\mu}( \mathbb D )} \frac{\lambda(w)} {( w - \tilde f_{\mu}(z) ) ^4}, 
		\]
	where 
		
		\[
			\lambda(w) = \frac {( \tilde f_{\mu} )_z} {\overline { (\tilde f_{\mu} )_z }} \cdot
			\frac{\tilde\nu}{1 -|\mu|} \circ \tilde f_{\mu}^{-1}. 
		\]
	
	That is, we get $\lambda_t = t\lambda + \delta(t)$ satisfies that $\| \delta(t) \|_{L^\infty} \to 0$ as $t \to 0$.
	From the chain rule of Schwarizian derivatives, we get
		\[
			F^{\mu} (f_{t \tilde \nu} ) = S ( g_t|_{\tilde f_{\mu}} ( \mathbb D^\ast ) ) ( \tilde f_{\mu} )_z '' + F^\mu(f_0), 
		\]
	so we obtain
		\[
			\lim_{t \to 0} \frac{1} {t} ( F^{\mu} (f_{t \tilde \nu}) - F^{\mu} (f_{0}) )
			= ( \tilde f_{\mu} )_z '' \int_{\tilde f_{\mu}( \mathbb D )} \frac{\lambda(w)} {( w - \tilde f_{\mu}(z) ) ^4}. 
		\]

	Finally, 
	by substituting $f(\zeta)$ for $w$ in the above integral, we obtain the representation formula \eqref{trans. by david}.
	Note that from the assumption that $\tilde f_{\mu}$ is an element of $W^{1,2}_{\loc}$, 
	substitution by $\tilde f_{\mu}$ can be computed using the Jacobian of $\tilde f_{\mu}$, and that the Jacobian of $f$ is
		\[
			| ( \tilde f_{\mu} )_{z}| ^2 - |(\tilde f_{\mu})_{\bar z} ) |^2. 
		\]

\end{proof}

\subsection{The derivative of a partial deformation} \label{tangentmap}

Next, we will compute the dimension of the image of
the tangent map $d_0 P^0$ of $\Bel( R_1 ) \ni \tau \mapsto S (f_{0 + \tau} ) \in \Cl ( \Teich (R) )$ at $0$. 
%Since $f_{\mu + \tau}$ is in $W_{\loc}^{1,2}(\mathbb C )$ from the definition of David deformation \eqref{def:david beltrami coe} and Theorem \ref {principle sol}, the tangent map is represented by 
%%	
%	\begin{equation} \label{eq:deriofbersembedd}
%		d_0 P ^\mu [\nu] ( z ) = - \frac {6} {\pi} ( f_{\mu + 0} )_ z ( z ) ) ^2 
%		\int_{\mathbb D} \frac {\nu ( \zeta ) } {(f_{\mu + 0} ( z ) - f_{\mu + 0} ( \zeta ) )^4}
%		 ( f_{\mu + 0} ) _ z ( \zeta ) ) ^2 
%	\end{equation}
%%	
%in the same manner as the calculation for Bers embedding. 
%In other words, since $f_{\mu + 0}$ belongs to $W_{\loc}^{1,2}(\mathbb C )$, the coordinate transformation by the function $f_{\mu + 0} $ can be done by multiplying the Jacobian of $f_{\mu+0}$ by the function under integration. 
Because the right side of the above equation \eqref{eq:deriofbersembedd} is pushed by $f_\mu$, 
let us begin by calculating the dimension of the image of $d_0 P^0$: 
	\[
		d_0 P^0 : T_0 ( \Bel ( R_1 ) ) \to T_{[0]} ( \Teich (R) ). 
	\]
As the kernel of $d_0P^0$ is expressed by
	\[
		N ( R_1, R ) : = \left\{ \nu \in L^\infty (R_1) \middle | \int \nu \varphi = 0 \ ( \forall \varphi \in Q(R_1, R ) ) \right\}, 
	\]
where 
	\[
		Q(R_1, R ) : = \{ \varphi \in Q(R_1) \mid \exists \Phi \in Q(R) \ \mathrm{s.t.} \ \Phi|_{R_1} = \varphi \}. 
	\]
Therefore, we only need to compute the dimension of 
	\begin{equation} \label{sp:wewant}
		T_0 ( \Bel ( R_1 ) ) / N ( R_1, R ) = L^\infty ( R_1 ) / N ( R_1, R ). 
	\end{equation}
To compute the dimension of the above space \ref{sp:wewant}, we use the well-known fact: 

\begin{theo} [{\cite[Theorem 7.5]{IT}}]

 	The mapping $\Lambda : \Bel(R_1) \to Q^\ast(R_1)$ given by $\Lambda_\mu$: 
		\[
			\Lambda_\mu : Q(R_1) \ni \varphi \mapsto \int \mu \varphi \in \mathbb C
		\]
	induces an isomorphism of $L^\infty ( R_1 ) / N ( R_1)$ onto $Q^\ast ( R_1 )$, where 
		\[
			N ( R_1) : = \left\{ \mu \in L^\infty(R_1) \middle| \int \varphi \mu = 0 \ ( \forall \varphi \in Q(R_1) ) \right\}. 
		\]

\end{theo}
Note that $N ( R_1, R)$ is a closed subspace in $N ( R_1)$, and $Q(R_1, R)$ is a closed subspace in $Q(R_1)$. 
Hence, there exists the natural surjection $\iota^\ast$ from $Q^\ast (R_1)$ to $Q^\ast(R_1, R)$. 
As above, from the universal property of the cokernel and the Hahn--Banach theorem, 
we get the following isomorphism in the sense of linear spaces: 
	\[\large
		\xymatrix{
			N ( R_1, R) / N ( R_1)
			\ar@{^{(}->}[r]
			& L^\infty (R_1) / N ( R_1) 
			\ar[d]^{ { \large \Lambda }}_{\text{\large {\rotatebox{90}{$\cong$} }} } 
			\ar[r]^-{\mathrm{coker}}
			& L^\infty(R_1) / N ( R_1, R)
			\ar[d]^{\exists{ \large \tilde \Lambda}}_{ \text{ \large \rotatebox{90}{$\cong$}}} \\
			\quad 
			& Q^\ast(R_1) \ar[r]^-{ { \large\iota^\ast} } 
			& Q^\ast(R_1, R) \ar@{}[lu]|{{\large \circlearrowright}}
		}
	\]
Hence, to prove that the image of $P^0$ is infinite--dimensional, we show that $\dim Q(R_1, R ) = \infty$ 
and $Q(R_1, R )$ becames a Banach space.  
That the dimension of  $Q(R_1, R )$ is infinity follows immediately from the fact that 
each $\varphi \in Q(R_1, R)$ has the unique extension of $\varphi$ to $R$. That is the following map is linear isomorphisms: 
	\[
		\eta : Q (R) \ni \Phi \mapsto \Phi |_{R_1} \in Q ( R_1, R )
	\]
Moreover, from the open mapping theorem in Functional analysis, 
we only prove that the inverse of $\eta$ is a bounded linear operator. 

\begin{prop} \label{genepudding}

 	If $R$ satisfies $\short ( R ) > 0$, then
	there is a constant $C_q$ depending on $R$ such that, 
		\[
			(\| \Phi \|_{L^1(R_1)} = ) \int_{R} | \Phi |
			\leq C_q(R) \int_{R_1} | \Phi | , 
		\]   
	for all integrable quadratic differentials $\Phi$ on $R$.

\end{prop}

The following lemma is a crucial complement to the proof of Proposition \ref {genepudding}. 

\begin{lem}[Puddings Lemma] \label{PuddingsLemma}

	Let $\mathbb A = \{ 1 < | z | < R \}$ and $\tilde { \mathbb A } = \{ r_1 < | z | < r_2 \}$ 
	with $1 < r_1 < r_2 < R$. Then, there exists a constant $C_a = C ( r_1, r_2, R ) > 0$ such that
		\[
			\| f \|_{L^1(\tilde {\mathbb A} )} \leq C_a \| f \|_{L^1(\mathbb A \setminus \tilde {\mathbb A} )}, 
		\] 
	for each a holomorphic function $f : \mathbb A \to \mathbb C$.

\end{lem}

\begin{proof}

 	Let $1 < r_0 < r_1$ and $ r_2 < r_3 < R$. Then, we get, for each 
	$z = r e^{ i \theta }$ in $\mathbb D_{r_2} \setminus \Cl ( \mathbb D_{r_1} )$, 
		\begin{eqnarray*}
			f ( z ) & = & 
			\frac {1} {2 \pi i} 
			\int_{\partial \mathbb D_{\alpha} + \partial \mathbb D_{\beta} }
			\frac {f(\zeta)} {\zeta - z} \dzeta, 
		\end{eqnarray*}
	where $r_0 < \beta < r_1, r_3 < \alpha < R$. 
	Hence, 	
		\begin{align*} 
			| f(z) | & \leq \frac {1} {2 \pi} \int_0 ^{2 \pi} \alpha 
			\frac {| f(\alpha e^{i \eta}) | } {| \alpha e^{i \eta}- r e^{i \theta} | }  d\eta 
			+ 
			\frac {1} {2 \pi} \int_0 ^{2 \pi} \beta
			\frac {| f(\beta e^{i \eta}) | } {| \beta e^{i \eta}- r e^{i \theta} | } d\eta, 
			\notag \\ 
			& \leq 
			\frac {1} {2 \pi} \frac {1} {\alpha - r} \int_0^{2 \pi} \alpha | f(\alpha e^{i \eta}) | d\eta
			+
			\frac {1} {2 \pi} \frac {1} {r -\beta} \int_0^{2 \pi} \beta | f(\beta e^{i \eta}) | d\eta. 
		\end{align*}
	Next, both sides are multiplied by $r$ and integrated. 
		\begin{align} \label{eq:pudding1}
			\int_{r_1}^{r_2} | f(z) | r \dr
			& \leq
			\frac {1} {2 \pi} \int_{r_1}^{r_2} \frac {r} {\alpha - r} \dr \int_0^{2 \pi} \alpha | f(\alpha e^{i \eta}) | d\eta
			+ 
			\frac {1} {2 \pi} \int_{r_1}^{r_2} \frac {r} {r -\beta} \dr \int_0^{2 \pi} \beta | f(\beta e^{i \eta}) | d\eta, \notag \\
			&\leq
			\frac {1} {2 \pi} \int_{r_1}^{r_2} \frac {r} {r_3 - r} \dr \int_0^{2 \pi} \alpha | f(\alpha e^{i \eta}) | d\eta
			+ 
			\frac {1} {2 \pi} \int_{r_1}^{r_2} \frac {r} {r - r_0} \dr \int_0^{2 \pi} \beta | f(\beta e^{i \eta}) | d\eta, \notag \\
			&\leq
			\frac {C(r_1, r_2, r_3)} {2 \pi}  \int_0^{2 \pi} \alpha | f(\alpha e^{i \eta}) | d\eta
			+ 
			\frac {C(r_0, r_1, r_2) } {2 \pi}  \int_0^{2 \pi} \beta | f(\beta e^{i \eta}) | d\eta, 
		\end{align}
	where 
		\[
			C(r_1, r_2, r_3) : = \int_{r_1}^{r_2} \frac {r} {r_3 - r} \dr, \ \ \ 
			C(r_0, r_1, r_2) : =  \int_{r_1}^{r_2} \frac {r} {r - r_0} \dr. 
		\]
	Third, integrate (\ref{eq:pudding1}) from 0 to 2$\pi$ for $\theta$. 
		\begin{equation} \label{eq:pudding2}
			\| f \|_{L^1(\tilde {\mathbb A} )} \leq 
			C(r_1, r_2, r_3) \int_0^{2 \pi} \alpha | f(\alpha e^{i \eta}) | d\eta
			+ 
			C(r_0, r_1, r_2) \int_0^{2 \pi} \beta | f(\beta e^{i \eta}) | d\eta. 
		\end{equation}
	Finally, integrate (\ref{eq:pudding2}) 
	from $r_3$ to $R$ with respect to $\alpha$ and from $r_0$ to $r_1$ with respect to $\beta$, we get
		\[
			\| f \|_{L^1(\tilde {\mathbb A} )} \leq 
			\left( C(r_1, r_2, r_3) ( r_1 -r_0 ) + C(r_0, r_1, r_2) (R - r_3 ) \right) 
			\| f \|_{L^1(\mathbb A \setminus \tilde {\mathbb A} )}. 
		\]

	To make a constant $C$ that depends on $1, r_1, r_2,$ and $R$, we can take $r_0$ and $r_3$ as, for example, 
	$\frac{1+r_1}{2}$ and $\frac{r_2 + R}{2}$, respectively. Further calculations lead to
		\begin{align*}
			& C(r_1, r_2, r_3) ( r_1 -r_0 ) + C(r_0, r_1, r_2) (R - r_3 )  \\
			& = -\frac{r_1 - 1}{2} \cdot \frac {r_2 + R} {2} \log \frac {r_3 - r_2} {r_3 - r_1} 
			+ \frac {R-r_2} {2} \cdot \frac {1 + r_1} {2} \log \frac {r_2 - r_0} {r_1 - r_0} \\
			& \leq R^2 \log \frac {2r_2 - 1 - r_1} {2r_1 - 1 - r_1} 
			\leq R^2 \log \frac {2r_2 } {r_1 - 1}. 
		\end{align*}
	Therefore, $C_a : = R^2 \log \frac {2r_2 } {r_1 - 1}$.

\end{proof}

\begin{Rem}
 
 	The constant $C$ in the above Proposition \ref {PuddingsLemma} depends 
	on the modulus of the annulus $\mathbb A$ and the annulus $\tilde {\mathbb A}$ and their relative positions. 
	For example, if the core curves of $\mathbb A$  and $\tilde { \mathbb A }$ coincide, 
	we know it depends on the modulus ratio. 
 
\end{Rem}

\begin{proof}[proof of Propotion \ref{genepudding}]

 	For each $m$, let $\Cyl_m$ be the collar of $\gamma_m^\ast$, then $R_1$ contains $\Cyl_m \setminus \widetilde C_m$. 
	From Proposition  \ref {PuddingsLemma}, 
	by using local coordinates if necessary, there exists a constant $C_a$ such that 
		\[
			\| \Phi \|_{L^1(\widetilde {C_j})} \leq C_a \| \Phi \|_{L^1(C_j \setminus \widetilde {C_j})}, 
		\]
	for each $j \in \mathbb N$. Therefore, we get
		\begin{align*}
			\int_R | \Phi | & = \left( \int_{R_1} + \int_{R_2} \right) | \Phi | 
			= \int_{R_1} | \Phi | + \sum_{j} \int_{\widetilde{C_j}} | \Phi |, \\
			& \leq  \int_{R_1} | \Phi | + C_a \sum_{j} \int_{\Cyl_j \setminus \widetilde{C_j}} | \Phi | 
			= \int_{R_1} | \Phi | + C_a \int_{\bigcup \Cyl_j \setminus \widetilde{C_j}} | \Phi |, \\
			& \leq \int_{R_1} | \Phi | + C_a \int_{R_1} | \Phi | = ( C_a + 1 ) \int_{R_1} | \Phi | . 
		\end{align*}
	Thus, we define $C_q$ as $C_a+1$.

\end{proof}

Hence, the following theorem is proved. 
	
	\begin{theo} [ ]

	 	Let $R$ be a Riemann Surface of analytically infinite type with $\short (R) > 0$ and 
		infinitely many essential closed geodesics $\{\gamma_n^\ast\}_{n \geq 0}$ whose lengths are bounded. 
		Then, the following operator

			\[	
				P^0 : \Bel ( R_1 ) \ni \tau \mapsto S ( f_{ 0 + \tau } ) \in B(R)
			\]
		has an infinite-dimensional image, where $R_1$ and $R_2$ are defined above.

	\end{theo}

This completes the proof of Theorem \ref{resultball}.

\section{Comparison with previous studies}\label{comp}

\begin{theo} [{\cite[Theorem 3.1]{Yamamoto}}] \label{Yamamoto}

	 	Let $R \cong \mathbb H / \Gamma$
		be a Riemann surface whose Fuchsian group is finitely generated Fuchsian group of the second kind.
		If $\psi \in \partial_{Bers} \Teich (S)$ is not a cusp, then there exists a finitely generated quasi-Fuchsian group
		of the first kind $G$ such that the limit set of $G$ contains the limit set of $\chi_\psi ( \Gamma )$.

	\end{theo}

	Theorem \ref{Yamamoto} is close to Theorem \ref {resultball}, 
	but the assumptions of the second kind and finite generation are different.
	In the case of the second kind, the free boundary part creates a gap in the limit set, 
	and by connecting it appropriately, a first-kind finitely generated quasi-Fuchsian group is constructed.

Compare the following assertions about finite type Riemann surfaces and infinite generated Fuchs groups. 
In an earlier work \cite{Th}, the following assertion was considered: 
	\begin{quote}
	 	Let $R$ be a Riemann surface such that the Fuchsian group of $R$ is infinitely generated of the first kind. 
		Then, there exists a 1-dimensional analytic subset in $\partial_{Bers} \Teich ( R )$
		 consisting of Kleinian groups whose limit set is homeomorphic to $\mathbb S^1$. 
		 That is, such Kleinian groups are \SQF s. 
	\end{quote}
The statement is similar to our Theorem \ref{resultball}, 
but in the present article, we have chosen a different way of proof to establish Theorem \ref{resultball}. 
The main reason is that a claim \cite[Propositon 4.4]{Th} on the existence of non-tangential limits for holomorphic maps from 
$\mathbb D $ to $\Teich(R)$, utilized in the earlier work, has a counterexample below. 
Therefore, the author must figure out how to prove the aforementioned statement now. 

\begin{exa}

	Let $f: \mathbb D \to \Teich ( \mathbb D )$ be a holomorphic map into the universal Teichm\"uller space
	defined by $f ( \lambda ) = [ \lambda ]$. Here, $\lambda \in \mathbb D$ is identified with the function on $\mathbb D$
	taking constant value $\lambda$. Then, $f$ does not have a non-tangential limit at any point in $\partial \mathbb D$.

\end{exa}
	
	\begin{proof}

	 	For each $\lambda \in \mathbb D $, the homepmprhism $f_\lambda$ on $\hat {\mathbb C}$ corespond to $\lambda$
		the quasiconformal map $z \mapsto z + \lambda \overline{z}$ for $z \in \mathbb D$ and
		$z \mapsto z + \lambda / z$ for $z \in \mathbb C \setminus \Cl ( \mathbb D )$. 
		Hence its schwarizian derivative $S ( f_\lambda )$ is 
			\[
				 \frac {-6 \lambda} {(z^2 - \lambda ) ^2}. 
			\]
	 	We get 
		\begin{align*}
			\| S(f_\lambda) - S( f_1 )\|_{B(\mathbb D )} 
			& = \sup_{z \in \hat{\mathbb C} \setminus \Cl( \mathbb D )} | ( S ( f_\lambda ) - S ( f_1 ) ) ( | z |^2 - 1 ) ^2 | \\
			& = \sup_{z \in \hat{\mathbb C} \setminus \Cl ( \mathbb D )} 
			\left| \left(  \frac {-6 \lambda} {(z^2 - \lambda ) ^2} -  \frac {-6} {(z^2 - 1) ^2} \right) 
			( | z |^2 - 1 ) ^2 \right| \\
			& = 6 \sup_{z \in \hat{\mathbb C} \setminus \Cl ( \mathbb D )} 
			\left| \left(  \frac {\lambda} {(z^2 - \lambda ) ^2} -  \frac {1} {(z^2 - 1) ^2} \right) 
			( | z |^2 - 1 ) ^2 \right| \\
			& = 6 \sup_{z \in \hat{\mathbb C} \setminus \Cl ( \mathbb D )} 
			\left| \left(  \frac {\lambda(z^2 - 1) ^2 - (z^2 - \lambda) ^2} {(z^2 - \lambda ) ^2(z^2 - 1) ^2} \right) 
			( | z |^2 - 1 ) ^2 \right| \\
			& = 6 \sup_{z \in \hat{\mathbb C} \setminus \Cl ( \mathbb D )} 
			\left| \left(  \frac {( \lambda -1 ) ( z^4 - \lambda )} {(z^2 - \lambda ) ^2(z^2 - 1) ^2} \right) 
			( | z |^2 - 1 ) ^2 \right| \\
			& = 6 | \lambda -1 | \sup_{z \in \hat{\mathbb C} \setminus \Cl ( \mathbb D )} 
			\left| \frac { z^4 - \lambda } {(z^2 - \lambda ) ^2} \right| 
			\left| \frac { ( | z |^2 - 1 ) ^2 } { (z^2 - 1) ^2 } \right| \\
			& \geq 6 | \lambda -1 | \sup_{x > 1} 
			\left| \frac {x^4 - \lambda } {(x^2 - \lambda ) ^2} \right|. 
		\end{align*}
	Moreover, $h_\lambda ( x ) : = \frac {x^4 - \lambda } {(x^2 - \lambda ) ^2}$ ($\lambda \in [ 0 ,1 ), x \in ( 1 , \infty )$)			is bounded below by $6$. Indeed, since
		\begin{align*}
			h_\lambda ' ( x ) 
			& = \frac {4x^3(x^2 - \lambda)^2 - 4x(x^4 -\lambda )(x^2 - \lambda )} {(x^2 - \lambda )^4} \\
			& = \frac {4\lambda x ( 1 - x^2 )} {(x^2 - \lambda )^3}, 
		\end{align*}
		 $h_\lambda$ has extrema at $x = 0 , \pm 1$. From $h' ( x ) < 0\ ( \forall x > 1 )$, it follows that 
		  $h_\lambda$ has maximum value at $x = 1$, hence
			\[
				\sup_{x > 1} | h_\lambda ( x ) |
				\geq 6 | \lambda -1 | \sup_{x > 1}  = | h_\lambda ( 1 ) | = \left | \frac {1} {1 - \lambda} \right|. 
			\]
		Summing up the above, 
			\[
				\| S(f_\lambda) - S( f_1 ) \|_{B(\mathbb D )} \geq 6 | \lambda -1 | \sup_{x > 1}
				 \left| \frac {x^4 - \lambda } {(x^2 - \lambda ) ^2} \right| = 6. 
			\]
		Finally, foe each $e^{i \theta} \in \partial \mathbb D$,  composing the rotation
		$z \mapsto  e^{- i \theta} z$ with  $f_{re^{i \theta }}$, we prove that
			\[
				\| S( f_{re^{i \theta }} ) -  S ( f_{e ^{i \theta}} ) \|_{B(\mathbb D )} \geq 6
			\]
		in the same manner.

	\end{proof}

%\Addresses


\begin{thebibliography}{Tani}

\bibitem[Abi]{Abi}
W.\ Abikoff,\ On boundaries of Teichm\"uller spaces and on Kleinian groups III,
Acta Math., \ \textbf{134} (1975), 211--234. 

\bibitem[AIM]{AIM} 
Kari Astala, Tadeusz Iwaniec, and Gaven Martin,
\textit{ Elliptic Partial Differential Equations and Quasiconformal Mappings in the Plane},
Princeton University Press, 2009.

\bibitem[Ahl]{QC}  Lars V. Ahlfors with appendix by C.\ J.\ Earle, I.\ Kra, M.\ Shishikura and J.\ H.\ Hubbard,
\textit{Lectures on Quasiconformal Mappings} Second Edition (American Mathematical Society, 2006)

\bibitem[Ber]{B}
L. \ Bers, 
On boundaries of Teichm\"uller spaces and on Kleinian groups I,\ 
Ann. of Math. \textbf{91} (1970), 570--600.


\bibitem[Dav]{D} G. David, Solutions de l'equation de Beltrami avec $\| \mu \| = 1$, Acta Math. \textbf{173}(1994) 37--60.

%\bibitem[Gar]{G} Frederick.\ P.\ Gardiner, 
%\textit{Teichm\"uller Theory and quadratic differentials}, John Wiley and Sons, 1987.
%
%\bibitem[Gre]{Gre} L.\ Greenberg, 
% Fundamental polyhedra for Kleinian groups, Ann.\ of Math., \textbf{84}(1966), 433--441


\bibitem[Hub]{H}
J.\ H.\ Hubbard, 
\textit{Teichm\"uller Theory and Applications to Geometry, Topology, and Dynamics}. 
Vol. 1. 
Matrix Editions, Ithaca, NY, 
2006. 

\bibitem[IT]{IT} 
Y.\ Imayoshi and M.\ Taniguchi, \textit{An introduction to Teichm\"uller spaces}, Springer -- Verlag, Tokyo, 1992.

\bibitem[Mas]{Mas}
B. \ Maskit, 
On boundaries of Teichm\"uller spaces and on Kleinian groups II,
Ann. of Math. \textbf{91} (1970), 607--639.

%\bibitem[Mat]{main} R.\ Matsuda,
%\textit{Construction of geodesics on Teichm\"uller spaces of Riemann surfaces with $\mathbb Z$ action}, 
%arXiv: 2022.03290, (2022).

\bibitem[McM]{McM} C. McMullen, Cusps are dense, Ann of Math. (2) \textbf{133} (1991), no. 1, 217--247.

\bibitem[Shi1]{Shi1}
H.\ Shiga, 
On analytic and geometric properties of Teichm\"uller spaces,
J.\ Math.\ Kyoto Univ.\ \textbf{24} (1984),\ 441--452.

\bibitem[Shi2]{Shi2}
H.\ Shiga, 
Characterization of quasidisks and Teichm\"uller spaces,
T\^ohoku Math.\ J.\ \textbf{37} (1985),\ 541-552.

\bibitem[Tan1]{Th}
H.\ Tanigawa,
Holomorphic families of geodesic discs in infinite dimensional Teichm\"uller spaces, 
Nagoya Math. J. \textbf{127} (1992), 117--128. 

\bibitem[Tan2]{Th2}
H.\ Tanigawa,
Holomorphic mappings into Teichm\"uller spaces,
Proc.\ Aemr.\ Math.\ Soc.\ \textbf{117} (1993), 71--78. 


\bibitem[Yam]{Yamamoto} H.\ Yamamoto, 
Boundaries of the Teichm\"uller spaces of finitely generated Fuchsian groups of the second kind,
J.\ Math.\ Soc.\ Japan, Vol.\ \textbf{35}(1983), No.\ 2,  307--321.

\bibitem[Wol]{W} S.\ A.\ Wolpert, The length spectra as moduli for compact Riemann Surfaces, Ann. of Math. (2), 
\textbf{109} (1979), 323--351.

\end{thebibliography}
\end{document}